\documentclass[12pt,reqno]{amsart}
\usepackage{amsmath, amsthm, amssymb, stmaryrd}

\usepackage{todonotes} 

\usepackage{hyperref}

\usepackage{cleveref}
\usepackage{bbm}
\usepackage{mathrsfs}
\usepackage{makecell}
\usepackage{geometry} 
\usepackage{xcolor}

\usepackage[backend=biber,style=numeric,sorting=nyt,giveninits=true]{biblatex}
\usepackage{graphicx}

\crefformat{section}{\S#2#1#3} 
\crefformat{subsection}{\S#2#1#3}
\crefformat{subsubsection}{\S#2#1#3}

\newtheorem{theorem}{Theorem}[section]
\newtheorem{lemma}[theorem]{Lemma}

\theoremstyle{definition}
\newtheorem{definition}[theorem]{Definition}

\theoremstyle{remark}

\numberwithin{equation}{section}

\newcommand{\codim}[1]{\operatorname{codim}\left(#1\right)}
\newcommand{\Sing}[1]{\operatorname{Sing}\left(#1\right)}

 \def\bfL{{\mathbf L}}
 \def\bfM{{\mathbf M}}

\def\bfr{{\mathbf r}}

\def\bfu{{\mathbf u}}
\def\bfv{{\mathbf v}}
\def\bfw{{\mathbf w}}
\def\bfx{{\mathbf x}}
\def\bfy{{\mathbf y}}
\def\bfz{{\mathbf z}}

\def\bfX{{\mathbf X}}
\def\bfY{{\mathbf Y}}

\def\calG{{\mathcal G}}

\def\calN{{\mathcal N}}

\def\C{{\mathbb C}}

\def\N{{\mathbb N}}  
\def\R{{\mathbb R}}
\def\Z{{\mathbb Z}}\def\Q{{\mathbb Q}}
\def\A{{\mathbb A}}

\def\grm{{\mathfrak m}}\def\grM{{\mathfrak M}}
\def\grN{{\mathfrak N}}\def\grn{{\mathfrak n}}

\def\grS{{\mathfrak S}}

\def\a{{\alpha}}  
\def\b{{\beta}}  
\def\gam{{\gamma}} 
\def\bfgam{{\boldsymbol \gam}}

\def\bfeta{{\boldsymbol \eta}}

\def\bfrho{{\boldsymbol \rho}}

\def \id{\text{\textbf{Id}}}

\def\eps{\varepsilon}
\def\bfeps{{\boldsymbol \varepsilon}}

\def\le{\leqslant} \def\ge{\geqslant}

\def\d{{\,{\rm d}}}

\begin{document}
\title[On weighted forms in many variables]{On weighted forms in many variables}
\author[Daniel Flores Galiote]{Daniel Flores Galiote}
\author[Kiseok Yeon]{Kiseok Yeon}
\address{Department of Mathematics, Rice University}
\email{df66@rice.edu}
\address{Department of Mathematics, University of California, Davis, United States}
\email{kyeon@ucdavis.edu}
\date{}
\dedicatory{}

\begin{abstract}
    In this paper, we introduce several novel approaches utilizing the circle method to obtain 
    the asymptotic formula for the number of integral points of bounded height lying on a 
    hypersurface in a weighted projective space.
    
    Let $F(\bfx;\bfy)$ be a given weighted form of degree $d$ in variables 
    $\bfx\in \mathbb{R}^{s_1}$ and $\bfy\in \mathbb{R}^{s_2}$, where variables $\bfx$ and $\bfy$ 
    have weights $w_1$ and $w_2$ with $w_1< w_2$, $(w_1,w_2)=1$, and
    $d>w_1w_2$. Write 
    \begin{equation*}
        R_F(P):=\#\{(\bfx;\bfy)\in \mathbb{Z}^{s_1+s_2}:\ F(\bfx;\bfy)=0,\ 
        |\bfx|\leq P^{w_1/d},\ |\bfy|\leq P^{w_2/d}\}.
    \end{equation*}
   In particular, we show that whenever 
   \[
    s_1+s_2-\sigma_F>
    \left(1+\frac{w_2}{w_1}\right)
    \frac d{w_1}2^{d/w_1},
    \]
    where $\sigma_F$ is the dimension of the affine singular locus of $F$, the quantity 
   $R_F(P)$ has the expected asymptotic formula, that is 
   \[R_F(P)=cP^{s_1w_1/d+s_2w_2/d-1}+o(P^{s_1w_1/d+s_2w_2/d-1}),\]
    where $c$ is the product of local densities. Furthermore, the constant $c$ is positive whenever $F(\bfx;\bfy)=0$ has a nonsingular solution over $\mathbb{R}$ and $\Q_p$ for every prime $p$. As a corollary, we verify the integral Hasse principle for the quasi-smooth hypersurface defined by $F(\bfx;\bfy)=0$ in a weighted projective space of sufficiently large dimensions.
\end{abstract}
\maketitle

\section{Introduction and statements of the main theorems}

This paper is concerned with asymptotic formulas for the number of integer
solutions of bounded height to weighted homogeneous polynomial equations in
many variables.  The circle method has proved particularly effective for
studying homogeneous equations when the number of variables is sufficiently
large relative to the degree.  In his celebrated work, Birch
\cite{Birch1962} established an asymptotic formula for the number of
integer solutions of bounded height to an equation $F(\bfx)=0$, where $F$ is a
degree-$d$ form in $s$ variables, under the condition
\[
    s-\sigma>(d-1)2^d,
\]
where $\sigma$ is the dimension of the affine singular locus associated
with $F$.  

Subsequent refinements, often formulated for counting functions
with smooth cutoffs, have improved this variable bound in several settings;
see the work of Heath-Brown \cite{MR703978,MR1421949}, Browning and
Prendiville \cite{Browning2017}, and Marmon and Vishe \cite{Marmon2019}.
Further improvements are possible when the form has additional structure.
For bihomogeneous forms, Schindler \cite{Schindler2014,Schindler2016}
developed a version of the circle method in which the relevant hypotheses
depend on the singular loci associated separately with the two sets of
variables. In the case of additive forms, Wooley
\cite[Theorem 1.3]{Wooley2012} proved that an asymptotic formula holds
provided that
\[
    s>d^2-\left\lceil\frac{\log d}{\log 2}\right\rceil.
\]

We now turn our attention to the non-homogeneous case. The current machinery of 
the circle method remains insufficient to establish an asymptotic formula for the 
number of integer solutions of bounded height to general polynomial equations $F(\bfx)=0$ 
in many variables. Indeed, the general approaches to homogeneous equations of high degrees, 
combined with the circle method, rely on differencing arguments to reduce the degrees, 
but this itself has a fundamental obstacle when handling the lower-degree terms in the 
polynomial. 

It is apparent that we require new machinery to access the non-homogeneous 
cases. Even for structured non-homogeneous equations, with the exception of additive 
forms, the literature is quite sparse compared to the extensively studied homogeneous 
cases. For example, non-homogeneous equations of low degree have been studied by Davenport and Lewis
\cite{DavenportLewis1964}, Watson \cite{Watson1967}, Browning and Heath-Brown
\cite{BrowningHeathBrown2009}, Flores \cite{Flores2024}, and Bernert \cite{Bernert2023}.

In the present work, we develop a version of the circle method specifically
tailored to weighted forms.  We establish an asymptotic formula for the
number of integer solutions to weighted homogeneous equations when the
number of variables is sufficiently large in terms of the weighted degree,
the weights, and the dimension of the singular locus, in the spirit of
Birch's landmark work \cite{Birch1962}.  We also obtain sharper quantitative
bounds when additional geometric information about the bihomogeneous
components is available.  Before stating the main results, we introduce the
relevant definitions and notation.

\begin{definition}\label{def:weightedform}
Let $\bfw=(w_1,\ldots,w_s)\in\N^s$ satisfy
$\gcd(w_1,\ldots,w_s)=1$.  A polynomial
$F\in\Z[x_1,\ldots,x_s]$ is an \emph{integral weighted form} of weighted
degree $d$ with weight vector $\bfw$ if
\[
F(\lambda^{w_1}x_1,\ldots,\lambda^{w_s}x_s)
    =\lambda^dF(x_1,\ldots,x_s),
\]
for every $\lambda\in\C$.  We then say that $x_i$ has weight $w_i$.
\end{definition}

Given an integral weighted form $F$ of degree $d$, our central problem is to
count the integer solutions to the equation
\[
    F(x_1,\ldots,x_s)=0,
\]
subject to the bounds
\[
    |x_i|\leq P^{w_i/d}\qquad(1\leq i\leq s).
\]
More precisely, we seek an asymptotic formula for this counting function as
$P\to\infty$, with a leading constant given by the expected product of
local densities.

In this paper, we restrict attention to weighted forms involving exactly two
distinct weights.  It seems plausible that the methods developed here may
generalize to more general weight vectors, but we do not pursue this
question.  Throughout the paper, let
$w_1,w_2,d,s_1,s_2\in\N$ satisfy
\[
    w_1<w_2,\qquad \gcd(w_1,w_2)=1,
    \qquad d>w_1w_2,
\]
and let $F(\bfx;\bfy)$ be a weighted form of degree $d$, where
$\bfx\in\C^{s_1}$ has weight $w_1$ and $\bfy\in\C^{s_2}$ has weight
$w_2$. The reason for the restriction $d>w_1w_2$ is discussed in
Subsection \ref{subsec:comparison with Birch}.

For $P\geq1$, define
\[
R_F(P):=\#\left\{(\bfx;\bfy)\in\Z^{s_1+s_2}:
\begin{array}{c}
F(\bfx;\bfy)=0,\ 
|\bfx|\leq P^{w_1/d},\ |\bfy|\leq P^{w_2/d}
\end{array}
\right\},
\]
and put
\begin{equation}\label{def of sigmaF}
\sigma_F:=\dim\left\{(\bfx;\bfy)\in\A_{\C}^{s_1+s_2}:
\nabla F(\bfx;\bfy)=\bf0\right\}.
\end{equation}
Our first theorem is intended to express the required number of variables in
terms of $\sigma_F$, $d$, and weights $w_1,w_2$.  To keep its hypothesis
readable, we state here a convenient sufficient bound.  A stronger version
with a more detailed hypothesis is given in Theorem
\ref{thm:weighted-singular-locus-precise}.
\begin{theorem}\label{thm:weighted-singular-locus}
Suppose that $d>w_1w_2$ and
\[
s_1+s_2-\sigma_F>
\left(1+\frac{w_2}{w_1}\right)
\frac d{w_1}2^{d/w_1}.
\]
Then, as $P\to\infty$, one has
\[
R_F(P)=\mathfrak S J\,P^{(s_1w_1+s_2w_2)/d-1}
+o\left(P^{(s_1w_1+s_2w_2)/d-1}\right),
\]
where $\mathfrak S$ and $J$ are the singular series and singular integral
associated with $F$. Moreover, $\mathfrak S J>0$ whenever
$F(\bfx;\bfy)=0$ has a nonsingular real solution and a nonsingular solution
over $\Q_p$ for every prime $p$. In particular, when $\sigma_F=0$, in other words, the
weighted hypersurface $V\subset \mathbb{P}(w_1,\ldots,w_1;w_2,\ldots,w_2)$ defined by $F(\bfx;\bfy)=0$ is quasi-smooth, the variety $V$ satisfies the integral Hasse
principle under the stated hypotheses.
\end{theorem}

Our other two theorems use geometric information about the individual
bihomogeneous components of $F$ and may give sharper bounds when the
corresponding singular loci are small. The possible bidegrees form the set
\[
    \mathcal D(d;w_1,w_2)
    :=\left\{(d_1,d_2)\in\Z_{\geq0}^2:
    w_1d_1+w_2d_2=d\right\}.
\]
Write its elements as
\[
    \mathcal D(d;w_1,w_2)
    =\left\{(d_1(k),d_2(k)):0\leq k\leq m\right\},
\]
where the first coordinates decrease with $k$. Coprimality of $w_1$ and
$w_2$ then gives
\[
    d_1(k+1)=d_1(k)-w_2,
    \qquad
    d_2(k+1)=d_2(k)+w_1.
\]
The weighted form $F$ has the unique decomposition
\begin{equation}\label{eqs:decomp}
    F(\bfx;\bfy)
    =\sum_{k=0}^{m}
    F^{(d_1(k),d_2(k))}(\bfx;\bfy),
\end{equation}
where $F^{(d_1,d_2)}$ is bihomogeneous of bidegree $(d_1,d_2)$ in
$(\bfx,\bfy)$. Let
\begin{equation}\label{def:sigma k}
    \sigma(k):=\dim\left\{(\bfx;\bfy)\in\A_{\C}^{s_1+s_2}:
    \nabla F^{(d_1(k),d_2(k))}(\bfx;\bfy)=\bf0\right\}.
\end{equation}
A stronger version of the following theorem is given in Theorem
\ref{thm:component-singular-locus-precise}.

\begin{theorem}\label{thm:component-singular-locus}
Suppose that $d>w_1w_2$ and that, for some $0\leq k\leq m$, one has
\[
    s_1+s_2-\sigma(k)
    >
    \left(1+\frac{w_2}{w_1}\right)
    2^{d_1(k)+d_2(k)-1}
    \bigl(d_1(k)+d_2(k)-1\bigr).
\]
Then, as $P\to\infty$, one has
\[
    R_F(P)
    =\mathfrak S J\,
    P^{(s_1w_1+s_2w_2)/d-1}
    +o\left(P^{(s_1w_1+s_2w_2)/d-1}\right),
\]
where $\mathfrak S$ and $J$ are the singular series and singular integral
associated with $F$. Moreover, $\mathfrak S J>0$ if the equation
$F(\bfx;\bfy)=0$ has a nonsingular solution over $\mathbb Q_p$ for every
prime $p$ and a nonsingular real solution.
\end{theorem}

For a bihomogeneous form $H(\bfx;\bfy)$ depending on both sets of
variables, let $V_1(H)$ and $V_2(H)$ denote the affine varieties defined
by $\nabla_{\bfx}H=\bf0$ and $\nabla_{\bfy}H=\bf0$, respectively, as in
Schindler \cite{Schindler2014,Schindler2016}. If $H$ depends only on
$\bfx$, we regard $V_1(H)$ as a subvariety of $\A_{\C}^{s_1}$; if it
depends only on $\bfy$, we similarly regard $V_2(H)$ as a subvariety of
$\A_{\C}^{s_2}$. For $0\leq k\leq m$, put
\begin{equation}\label{def:hat sigma}
\widehat{\sigma}(k):=
\begin{cases}
s_2+\dim V_1(F^{(d_1(k),0)}),&d_2(k)=0,\\[2mm]
\displaystyle\max_{i\in\{1,2\}}
\dim V_i(F^{(d_1(k),d_2(k))}),&d_1(k)d_2(k)>0,\\[2mm]
s_1+\dim V_2(F^{(0,d_2(k))}),&d_1(k)=0.
\end{cases}
\end{equation}
Thus $s_1+s_2-\widehat{\sigma}(k)$ is the relevant singular-locus
codimension for every $k$. A stronger version of the following theorem is
given in Theorem \ref{thm:V1-V2-criterion-precise}.

\begin{theorem}\label{thm:V1-V2-criterion}
Suppose that $d>w_1w_2$. For $0\leq k\leq m$, put
\[
    \iota_k:=
    \begin{cases}
        1,&d_1(k)d_2(k)>0,\\
        0,&d_1(k)d_2(k)=0.
    \end{cases}
\]
If, for some $0\leq k\leq m$, one has
\[
    s_1+s_2-\widehat{\sigma}(k)
    >
    \left(1+\frac{w_2}{w_1}\right)
    2^{d_1(k)+d_2(k)-1-\iota_k}
    \bigl(d_1(k)+d_2(k)-1\bigr),
\]
then, as $P\to\infty$, one has
\[
    R_F(P)
    =\mathfrak S J\,
    P^{(s_1w_1+s_2w_2)/d-1}
    +o\left(P^{(s_1w_1+s_2w_2)/d-1}\right),
\]
where $\mathfrak S$ and $J$ are the singular series and singular integral
associated with $F$. Moreover, $\mathfrak S J>0$ if the equation
$F(\bfx;\bfy)=0$ has a nonsingular solution over $\mathbb Q_p$ for every
prime $p$ and a nonsingular real solution.
\end{theorem}

\subsection{Exploiting the weighted structure}
\label{subsec:comparison with Birch}

We first explain the restriction $d>w_1w_2$. If $d<w_1w_2$, then
$\mathcal D(d;w_1,w_2)$ contains at most one bidegree. Indeed, two distinct
nonnegative solutions of $w_1d_1+w_2d_2=d$ differ by a nonzero integral
multiple of $(w_2,-w_1)$, which would force $d\geq w_1w_2$. Thus a
nonzero weighted form of degree $d<w_1w_2$ is either an ordinary
homogeneous form in one set of variables or a single bihomogeneous form.
These cases fall within the settings treated by Birch \cite{Birch1962} and
Schindler \cite{Schindler2014,Schindler2016}, respectively.

When $d=w_1w_2$, the only possible bidegrees are
\[
    (w_2,0),
    \qquad
    (0,w_1).
\]
Hence
\[
    F(\bfx;\bfy)
    =F^{(w_2,0)}(\bfx)+F^{(0,w_1)}(\bfy).
\]
The corresponding exponential sum factors into two homogeneous
exponential sums, and the circle method may be applied using the usual
homogeneous estimates for the two factors. The range $d>w_1w_2$ is
therefore the first range in which a weighted form may exhibit a more
complicated interaction among its bihomogeneous components, and it is the
range considered here.

Our results have two main advantages over the bound suggested by naively
treating a weighted form as an ordinary homogeneous form; we explain this
comparison below.  The first is that the number of differencing steps may be
determined by the total degree of a suitable bihomogeneous component, rather
than by the weighted degree of the original form, which may be much larger.
The second is that our methods respect the weighted structure of the form and
preserve the contribution of the variables of larger weight.

The guiding analytic principle behind our investigation of weighted forms is
that a variable of weight $w$ can occur with ordinary degree at most $d/w$.
Thus variables of larger weight occur with lower degree, and one expects to
establish stronger cancellation in the corresponding exponential sums with
fewer differencing steps.  A treatment based on an associated ordinary
homogeneous form obscures precisely this feature by raising the variables to
powers that make every monomial have ordinary degree $d$.  This principle was
first explored by the first author in \cite{Flores2024} for a structured
family of weighted quartic forms, in which the variables of larger weight
occur quadratically rather than quartically.  

To describe the comparison, consider the modified polynomial
\[
G(\bfx;\bfy):=F(\bfx^{w_1};\bfy^{w_2}),
\]
where the powers are taken coordinatewise.  This is an ordinary homogeneous
form of degree $d$. Put
\[
\sigma_G:=
\dim\left\{(\bfx;\bfy)\in\A_{\C}^{s_1+s_2}:
\nabla G(\bfx;\bfy)=\bf0\right\}.
\]
If one disregards the weighted structure and uses $G(\bfx;\bfy)$, then
Birch's theorem suggests a variable threshold of size
\[
    s_1+s_2-\sigma_G>(d-1)2^d.
\]
Note that this is only a heuristic and that this substitution does not identify the integer
solutions counted by $R_F(P)$ with all integer solutions of $G=0$.

We first illustrate the saving in the number of differencing steps. If
$H(\bfx;\bfy)$ depends on both sets of variables, then
\[
\A_{\C}^{s_1}\times\{\mathbf 0\}\subseteq V_1(H),
\qquad
\{\mathbf 0\}\times\A_{\C}^{s_2}\subseteq V_2(H).
\]
Thus the smallest possible dimensions of these varieties are $s_1$ and
$s_2$, respectively. A particularly favorable case occurs when $w_2\mid d$
and the singular locus of the pure $\bfy$-component
$F^{(0,d/w_2)}$ is zero-dimensional. Applying Theorem
\ref{thm:V1-V2-criterion-precise} to this component requires only
\[
s_2>
\max\left\{2,
\frac{dw_2}{w_1\bigl(d+w_2(w_2-w_1-1)\bigr)}\right\}
2^{d/w_2-1}\left(\frac d{w_2}-1\right).
\]
For example, under the additional assumption that $\sigma_G=0$, the following
table compares the resulting conditions.
\[
\begin{array}{c|c|c|c}
(w_1,w_2) & d & \text{Naive Birch condition}
& \text{Pure $\bfy$-component condition}\\ \hline
(1,2) & 4  & s_1+s_2>48    & s_2>4\\
(1,3) & 9  & s_1+s_2>4096  & s_2>18\\
(2,3) & 12 & s_1+s_2>45056 & s_2>48
\end{array}.
\]
An even more pronounced saving is visible in the family
$(w_1,w_2,d)=(1,W,2W)$.  The case $W=2$, for a particular family of
weighted quartic forms, was studied previously by the first author in
\cite{Flores2024}. In this family, the pure $\bfy$-component condition above
reduces to
\[
    s_2>4,
\]
whereas, if the associated ordinary homogeneous form
$G(\bfx;\bfy)=F(\bfx;\bfy^W)$ is nonsingular, a naive application of Birch's
theorem would lead to the condition
\[
    s_1+s_2>(2W-1)2^{2W}.
\]

This naive comparison from $F$ to $G$ need not preserve the dimension of the 
associated singular locus. As with everything in life, there is a cost to making
 our lives easier, here the cost is geometric in nature. Consider the family of 
 weighted forms satisfying
\[
    (w_1,w_2,d)=(1,W,2W),
    \qquad W>1.
\]
Put
\[
    G(\bfx;\bfy)=F(\bfx;\bfy^W),
\]
so that the chain rule gives
\[
\frac{\partial G}{\partial x_i}
=\frac{\partial F}{\partial x_i}(\bfx;\bfy^W),
\qquad
\frac{\partial G}{\partial y_j}
=W y_j^{W-1}
\frac{\partial F}{\partial y_j}(\bfx;\bfy^W),
\]
and consider the finite morphism
\[
    (\bfx;\bfy)\longmapsto(\bfx;\bfy^W).
\]
It preserves dimension, and its inverse image of the singular locus of $F$ is
contained in the singular locus of $G$.  Hence $\sigma_G\geq\sigma_F$.
Moreover, the factor $y_j^{W-1}$ permits the equation involving
$\partial F/\partial y_j$ to be lost on the coordinate hyperplane $y_j=0$.

Thus, it is not difficult to show that
\begin{equation}\label{eq:sigma G comparison}
    \sigma_F\leq\sigma_G\leq\sigma_F+s_2.
\end{equation}

Additionally, the upper bound in \eqref{eq:sigma G comparison} can be seen to be sharp by
considering, when $s_1=s_2=s$, the weighted form
\[
    F(\bfx;\bfy)=
    \sum_{j=1}^{s}\left(y_jx_j^W+y_j^2\right).
\]
For this example, one has $\sigma_F=0$ and $\sigma_G=s=s_2$.  Thus the
upper bound in \eqref{eq:sigma G comparison} is attained.

\subsection*{Acknowledgements}
The second author gratefully acknowledges support from the University of
California, Davis, through the KAP allocation, and thanks Rice University for supporting 
a visit that facilitated the completion of this work. The first author thanks Trevor
D. Wooley for suggesting this problem to him as his thesis problem, and thanks Rice 
University for its support. The authors also thank ChatGPT 5.6 Sol for suggesting 
the commutative-algebraic proof strategy used in Lemma \ref{lemma:codimension comparison}.
 The authors rewrote the argument and verified its mathematical accuracy 
 and rigor.

\subsection{Notational conventions}\label{subsec:notational conventions}

Unless stated otherwise, the parameters $X$, $Y$, and $P$ are positive and
tend to infinity.  Each assertion containing $\eps$ is understood to hold
for every $\eps>0$.  The implicit constants in Vinogradov's notation may
depend on $\eps$, $s_1$, $s_2$, $d$, the weights, and the coefficients of
$F$.

We use lower-case bold letters for vectors and upper-case bold letters for
matrices.  For either a vector $\bfu$ or a matrix $\mathbf U$, the symbols
$|\bfu|$ and $|\mathbf U|$ denote the maximum of the absolute values of the
entries.  Inequalities between vectors or matrices are interpreted
entrywise.

For $z\in\R$, let $\|z\|$ denote the distance from $z$ to the nearest
integer.  If $\bfu=(u_1,\ldots,u_r)$, put
\[
    \|\bfu\|:=\max_{1\leq i\leq r}\|u_i\|.
\]
We write $\{z\}=z-\lfloor z\rfloor$ for the fractional part of a real
number $z$.  For a vector $\bfu=(u_1,\ldots,u_r)$, fractional parts are
taken coordinatewise, so that
$\{\bfu\}=(\{u_1\},\ldots,\{u_r\})\in[0,1)^r$.
When $B\subset\R^r$ is a box, a sum over $\bfu\in B$ is understood to run
over $B\cap\Z^r$.  Finally, write $e(z)=e^{2\pi iz}$.

If the list of summation variables is empty, the corresponding sum is
omitted. Thus, when $i=0$, our convention is
\[
    \sum_{\bfx_1,\ldots,\bfx_i\in B} A=A.
\]

\section{Auxiliary lemmas}\label{sec:auxiliary lemmas}
We begin by recording several results
from Schmidt \cite{Schmidt1985} and Schindler \cite{Schindler2014}.
For vectors $\bfu_1,\ldots,\bfu_j\in\Z^r$, we use the tuple notation
\begin{equation}\label{eq:tuple notation}
    \mathbf U^{(j)}:=[\bfu_1,\ldots,\bfu_j]\in\Z^{r\times j}.
\end{equation}
In particular, we write
$\bfX^{(j_1)}=[\bfx_1,\ldots,\bfx_{j_1}]$ and
$\bfY^{(j_2)}=[\bfy_1,\ldots,\bfy_{j_2}]$, with their ambient dimensions
understood from context. When $j=0$, this notation represents the empty
tuple, in accordance with subsection \ref{subsec:notational conventions}.

For a
polynomial $F\in\R[x_1,\ldots,x_s]$ and an integer $j\geq1$, define
\[
F_j\bigl(\bfX^{(j)}\bigr)
:=\sum_{\boldsymbol\epsilon\in\{0,1\}^j}
(-1)^{\epsilon_1+\cdots+\epsilon_j}
F(\epsilon_1\bfx_1+\cdots+\epsilon_j\bfx_j).
\]
For a finite set $B\subset\Z^s$, let $B^D:=B-B$ and
\[
B\bigl(\bfX^{(j)}\bigr)
:=\bigcap_{\boldsymbol\epsilon\in\{0,1\}^j}
\left(B-(\epsilon_1\bfx_1+\cdots+\epsilon_j\bfx_j)\right).
\]
When $j=0$, we set $B(\bfX^{(0)}):=B$.

The following is a version of Weyl's inequality due to Schmidt.
\begin{lemma}\label{lemma:schmidtweyl}
    Let $F \in \R[x_1,\ldots,x_s]$ and put
    \[S = \sum_{\bfx \in B} e(F(\bfx)).\]
    Then, for every integer $j\geq1$,
    \[
    |S|^{2^{j-1}}
    \leq |B^D|^{2^{j-1}-j}
    \sum_{\bfx_1,\ldots,\bfx_{j-1}\in B^D}
    \left|\sum_{\bfx_j\in B(\bfX^{(j-1)})}
    e\bigl(F_j(\bfX^{(j)})\bigr)\right|.
    \]
\end{lemma}
\begin{proof}
    For $j\geq2$, see \cite[Lemma 11.1]{Schmidt1985}. When $j=1$, the tuple
    $\bfX^{(j-1)}$ is empty and $B(\bfX^{(0)})=B$. Moreover,
    $F_1(\bfx)=F(0)-F(\bfx)$, so the right hand side is $|S|$. Thus the
    asserted inequality is an identity in this case.
\end{proof}

\begin{lemma}\label{lemma:schindlerweyl}
    Let $F \in \R[x_1,\ldots,x_s]$ and put
    \[S = \sum_{\bfx \in B} e(F(\bfx)).\]
    Then, for every integer $j\geq2$,
    \[
    |S|^{2^{j-1}}\leq |B^D|^{2^{j-1}-j}
    \sum_{\substack{\bfx_1,\ldots,\bfx_{j-2}\in B^D \\
    \bfx_{j-1}\in B(\bfX^{(j-2)})^D \\
    \bfx_j\in B(\bfX^{(j-1)})}}
    e\bigl(F_j(\bfX^{(j)})-F_{j-1}(\bfX^{(j-1)})\bigr),
    \]
    where the sum on the right is a real non-negative value.
\end{lemma}
\begin{proof}
    For $j\geq3$, see the argument of \cite[\S 2]{Schindler2014}. When
    $j=2$, the asserted inequality is the usual first Weyl-differencing
    identity. Indeed,
    \[
        F_2(\bfX^{(2)})-F_1(\bfX^{(1)})
        =F(\bfx_1+\bfx_2)-F(\bfx_2),
    \]
    and the stated formula follows by expanding $|S|^2$ and writing
    $\bfx_1$ for the difference of the two summation variables.
\end{proof}

\begin{lemma}\label{lemma:schmidtdifferenceidentify}
    Suppose that $F$ is a form of degree $d>0$. Then the following hold.
    \begin{enumerate}
        \item[(a)] $F_j(\bfX^{(j)})=0$ when $j>d$.
        \item[(b)] $F_d(\bfX^{(d)})$ is multilinear.
        \item[(c)] When $1\leq j<d$, one has
        \[F_j(\bfX^{(j)})=\sum_{1\leq l\leq d-j+1}H_{j,l}(\bfX^{(j)}),\]
        where $H_{j,l}$ is homogeneous of degree $l$ in $\bfx_j$ and of
        total degree $d-l$ in $\bfx_1,\ldots,\bfx_{j-1}$.
    \end{enumerate}
\end{lemma}
\begin{proof}
    See \cite[Lemma 11.2]{Schmidt1985}.
\end{proof}

\begin{lemma}\label{lemma:bilinearbound}
    Let $r,s\geq1$, let $\bfM\in\R^{r\times s}$ and
    $\bfz\in\R^r$, and let $C_1,C_2>0$. Suppose that
    $B_1\subset[-C_1,C_1]^r$ and $B_2\subset[-C_2,C_2]^s$ are
    axis-parallel boxes of side length at most $2$. If $X\geq2$ and
    $Y\geq1$, then
    \[
    \sum_{\bfy\in YB_2}\left|\sum_{\bfx\in XB_1}
    e\bigl(\bfx^T(\bfM\bfy+\bfz)\bigr)\right|
    \ll\mathcal N(\bfM;X,Y)(X\log X)^r,
    \]
    where
    \[
    \mathcal N(\bfM;X,Y):=\#\{\bfy\in\Z^s:|\bfy|\leq2C_2Y,
    \ \|\bfM\bfy\|\leq X^{-1}\}.
    \]
    The implied constant may depend on $r,s,C_1$, and $C_2$.
\end{lemma}
\begin{proof}
    The one-dimensional geometric-sum estimate gives
    \begin{align*}
        \sum_{\bfy\in YB_2}\left|\sum_{\bfx\in XB_1}
        e\bigl(\bfx^T(\bfM\bfy+\bfz)\bigr)\right|\ll\sum_{\bfy\in YB_2}\prod_{i=1}^{r}
        \min\left\{X,\left\|\sum_{j=1}^{s}M_{i,j}y_j+z_i
        \right\|^{-1}\right\}.
    \end{align*}
    Put $L=\lceil X\rceil$ and partition $[0,1)^r$ into half-open
    cubes of side length $L^{-1}$. Let $N_{\bfr}$ count the vectors
    $\bfy\in YB_2$ for which $\{\bfM\bfy+\bfz\}$ lies in the cube
    indexed by $\bfr$. If $\bfy_0$ is one such vector, the injective map
    $\bfy\mapsto\bfy-\bfy_0$ gives
    $N_{\bfr}\leq\mathcal N(\bfM;X,Y)$, since
    $|\bfy-\bfy_0|\leq2C_2Y$ and
    $\|\bfM(\bfy-\bfy_0)\|\leq L^{-1}\leq X^{-1}$.

    On the cube indexed by $\bfr$, the product in the preceding sum is
    bounded by
    \[
        \ll X^r\prod_{i=1}^{r}
        \frac{1}{1+\min\{r_i,L-1-r_i\}}.
    \]
For each coordinate, we use the estimate
\[
    \sum_{r=0}^{L-1}\frac{1}{1+\min\{r,L-1-r\}}\ll\log X,
\]
and summing over $0\leq r_i<L$ proves the result.
\end{proof}

For $a>0$ and $\bfM\in\R^{s_1\times s_2}$, define
    \[
    \Lambda(a;\bfM):=
    \left[\begin{matrix}
        a^{-1} \id_{s_2} & \bf0 \\
        a \bfM & a \id_{s_1}
    \end{matrix}\right], 
    \]
    together with the associated counting functions
    \[
    U(Z):=\#\{\bfz\in\Z^{s_1+s_2}:|\Lambda(a;\bfM)\bfz|\leq Z\},
    \qquad
    U^t(Z):=\#\{\bfz\in\Z^{s_1+s_2}:|\Lambda(a;\bfM^T)\bfz|\leq Z\}.
    \]
    \begin{lemma}\label{lemma:lattice point comparison}
        Let $C_0>0$ be fixed. If $0<Z_1 \leq Z_2 \leq C_0$, then one has the bound
        \[
        U\left(Z_2\right) \ll \max \left(\left(\frac{Z_2}{Z_1}\right)^{s_2} U\left(Z_1\right), \frac{(aZ_2)^{s_2}}{(aZ_1)^{s_1}} U^t\left(Z_1\right)\right) ,
        \]
        where the implied constant may depend on $C_0,s_1$ and $s_2$.
    \end{lemma} 
    \begin{proof}
        This is the fixed-range version of \cite[Lemma 3.1]{Schindler2014}.
        Schindler's proof first reduces the argument to Euclidean
        lattice-point counts with $Z_2\leq\sqrt{s_1+s_2}$. The same proof
        works for any fixed upper bound $C_0$, with the extra dependence
        absorbed into the implied constant.
    \end{proof}

We conclude this section with a comparison between the singular locus of a
form and a variety defined by only some of its partial derivatives. Yamagishi
proved a similar comparison for the bihomogeneous form
$F(x_1y_1,\ldots,x_ny_n)$ associated with a homogeneous form $F$
\cite[Theorem 5.1]{Yamagishi2019}. Unfortunately, the proof presented by Yamagishi does not 
extend to general bihomogeneous forms. Here we present a general version 
of this statement, which holds for general multihomogeneous forms.

\begin{lemma}\label{lemma:codimension comparison}
    Let $G\in\C[z_1,\ldots,z_N]$ be a nonzero homogeneous form of degree at
    least $2$, and let $\mathcal E$ be a nonempty subset of
    $\{1,\ldots,N\}$. Suppose that there are $a_0\in\C^\times$ and
    polynomials
    \[
    A_i\in(z_1,\ldots,z_N)
    \qquad (i\in\mathcal E),
    \]
    for which the following partial Euler relation holds:
    \begin{equation}\label{eq:partial Euler relation}
        a_0G=\sum_{i\in\mathcal E}
        A_i\frac{\partial G}{\partial z_i}.
    \end{equation}
    Define
    \[
    V_{\mathcal E}(G):=
    \left\{\bfz\in\A_{\C}^{N}:
    \frac{\partial G}{\partial z_i}(\bfz)=0
    \text{ for every }i\in\mathcal E\right\},
    \]
    and
    \[
    \Sing{G}:=
    \left\{\bfz\in\A_{\C}^{N}:\nabla G(\bfz)=\mathbf 0\right\}.
    \]
    Then
    \[
    \codim{\Sing{G}}\leq 2\codim{V_{\mathcal E}(G)}.
    \]

    In particular, suppose that the variables are partitioned as
    \[
    \bfz=(\bfz^{(1)},\ldots,\bfz^{(r)}),
    \]
    and that $G$ is multihomogeneous of multidegree
    $(d_1,\ldots,d_r)$, where $d_j>0$ for every $j$. For $1\leq j\leq r$,
    let
    \[
    V_j(G):=
    \left\{\bfz\in\A_{\C}^{N}:
    \nabla_{\bfz^{(j)}}G(\bfz)=\mathbf 0\right\}.
    \]
    Then
    \[
    \codim{\Sing{G}}
    \leq 2\min_{1\leq j\leq r}\codim{V_j(G)}.
    \]
\end{lemma}

\begin{proof}
    Put $R=\C[z_1,\ldots,z_N]$ and define
    \[
    I_{\mathcal E}=
    \left(\frac{\partial G}{\partial z_i}:i\in\mathcal E\right),
    \qquad
    J=\left(
        \frac{\partial G}{\partial z_1},\ldots,
        \frac{\partial G}{\partial z_N}
    \right).
    \]
    We prove
    \[
    \operatorname{ht}J\leq 2\operatorname{ht}I_{\mathcal E}.
    \]

    Set $c=\operatorname{ht}I_{\mathcal E}$, and choose a homogeneous
    minimal prime $\mathfrak p$ over $I_{\mathcal E}$ such that
    \[
    \operatorname{ht}\mathfrak p=c.
    \]
    Since $\mathfrak p$ is a proper homogeneous prime, it is contained in
    the homogeneous maximal ideal
    \[
    \mathfrak n=(z_1,\ldots,z_N).
    \]
    Put
    \[
    A=(R/\mathfrak p)_{\mathfrak n/\mathfrak p},
    \qquad
    \mathfrak m=(\mathfrak n/\mathfrak p)A,
    \qquad
    M=(\mathfrak p/\mathfrak p^2)_{\mathfrak n/\mathfrak p}.
    \]
    Then $(A,\mathfrak m)$ is a Noetherian local domain and $M$ is a
    finitely generated $A$-module.

    Let $L=\operatorname{Frac}(A)$. Localizing at the generic point of
    $V(\mathfrak p)$ gives
    \[
    M\otimes_A L
    \cong
    \mathfrak pR_{\mathfrak p}/(\mathfrak pR_{\mathfrak p})^2.
    \]
    The local ring $R_{\mathfrak p}$ is regular, its maximal ideal is
    $\mathfrak pR_{\mathfrak p}$, and its residue field is
    \[
    R_{\mathfrak p}/\mathfrak pR_{\mathfrak p}
    \cong \operatorname{Frac}(R/\mathfrak p)=L.
    \]
    By the defining property of a regular local ring,
    \[
    \dim_L
    \frac{\mathfrak pR_{\mathfrak p}}
    {(\mathfrak pR_{\mathfrak p})^2}
    =\dim R_{\mathfrak p}
    =\operatorname{ht}\mathfrak p
    =c.
    \]
    Hence
    \[
    \operatorname{rank}_A M
    =\dim_L(M\otimes_A L)
    =c.
    \]

    Since $I_{\mathcal E}\subseteq\mathfrak p$, the relation
    \eqref{eq:partial Euler relation} shows that $G\in\mathfrak p$. Let
    \[
    u=G+\mathfrak p^2\in M.
    \]
    Passing \eqref{eq:partial Euler relation} to
    $\mathfrak p/\mathfrak p^2$ and then localizing yields
    \[
    u=
    \frac1{a_0}\sum_{i\in\mathcal E}
    \overline{A_i}
    \left(
        \frac{\partial G}{\partial z_i}+\mathfrak p^2
    \right)
    \in\mathfrak mM.
    \]

    For a finitely generated $A$-module $N$ and $v\in N$, define the order
    ideal of $v$ by
    \[
    N^*(v):=
    \{\varphi(v):\varphi\in\operatorname{Hom}_A(N,A)\}.
    \]
    The generalized principal ideal theorem
    for order ideals \cite[Theorem 2.2]{EisenbudHunekeUlrich2001} gives
    \[
    \operatorname{ht}_A N^*(v)
    \leq \operatorname{rank}_A N
    \qquad (v\in\mathfrak mN).
    \]
    Applying this with $N=M$ and $v=u$, we obtain
    \[
    \operatorname{ht}_A M^*(u)\leq c.
    \]

    For each $1\leq j\leq N$, differentiation with respect to $z_j$ induces an
    $R/\mathfrak p$-linear map
    \[
    \begin{aligned}
    \delta_j:\mathfrak p/\mathfrak p^2&\longrightarrow R/\mathfrak p,\\
    h+\mathfrak p^2&\longmapsto
    \frac{\partial h}{\partial z_j}+\mathfrak p.
    \end{aligned}
    \]
    Indeed,
    \[
    \frac{\partial(\mathfrak p^2)}{\partial z_j}\subseteq\mathfrak p,
    \]
    so the map is well defined. Moreover, if $a\in R$ and
    $h\in\mathfrak p$, then
    \[
    \frac{\partial(ah)}{\partial z_j}
    =
    a\frac{\partial h}{\partial z_j}
    +
    h\frac{\partial a}{\partial z_j}
    \equiv
    a\frac{\partial h}{\partial z_j}
    \pmod{\mathfrak p},
    \]
    so the map is $R/\mathfrak p$-linear. After localization,
    $\delta_j\in\operatorname{Hom}_A(M,A)$, and therefore
    \[
    \frac{\partial G}{\partial z_j}+\mathfrak p
    =\delta_j(u)\in M^*(u).
    \]

    Let $\overline J$ be the image of $J$ in $A$. The preceding argument
    shows that the image of every generator of $J$ belongs to $M^*(u)$.
    Thus
    \[
    \overline J\subseteq M^*(u).
    \]
    Choose a prime $\overline{\mathfrak q}$ of $A$ containing $M^*(u)$
    such that
    \[
    \operatorname{ht}_A\overline{\mathfrak q}\leq c,
    \]
    and let $\mathfrak q$ be its inverse image in $R$. Then
    \[
    J\subseteq\mathfrak q,
    \qquad
    \mathfrak p\subseteq\mathfrak q\subseteq\mathfrak n.
    \]
    Since polynomial rings over fields are catenary,
    \[
    \operatorname{ht}_R\mathfrak q
    =
    \operatorname{ht}_R\mathfrak p
    +
    \operatorname{ht}_{R/\mathfrak p}
    (\mathfrak q/\mathfrak p).
    \]
    Localization at $\mathfrak n/\mathfrak p$ does not change the height
    of $\mathfrak q/\mathfrak p$, so
    \[
    \operatorname{ht}_{R/\mathfrak p}(\mathfrak q/\mathfrak p)
    =
    \operatorname{ht}_A\overline{\mathfrak q}
    \leq c.
    \]
    Therefore,
    \[
    \operatorname{ht}_R\mathfrak q\leq 2c.
    \]
    Since $J\subseteq\mathfrak q$, it follows that
    \[
    \operatorname{ht}J
    \leq\operatorname{ht}\mathfrak q
    \leq2c
    =2\operatorname{ht}I_{\mathcal E}.
    \]

    Since
    \[
    \codim{V_{\mathcal E}(G)}=\operatorname{ht}I_{\mathcal E},
    \qquad
    \codim{\Sing{G}}=\operatorname{ht}J,
    \]
    this proves the first assertion.

    Now suppose that $G$ is multihomogeneous as in the final part of the
    statement. Write
    $\bfz^{(j)}=(z^{(j)}_1,\ldots,z^{(j)}_{N_j})$. For every
    $1\leq j\leq r$, Euler's identity in the variables $\bfz^{(j)}$ gives
    \[
    d_jG=\sum_{i=1}^{N_j}
    z^{(j)}_i\frac{\partial G}{\partial z^{(j)}_i}.
    \]
    Hence \eqref{eq:partial Euler relation} holds with $\mathcal E$ equal
    to the indices of the variables in the $j$th set. Applying the first
    assertion for each $j$ and taking the minimum proves the final assertion.
\end{proof}

\bigskip

\section{The circle method}\label{sec:circle method}
Let $P$ be a large parameter and set $X=P^{w_1/d}$ and $Y=P^{w_2/d}$.
We write
\[
S(\a):=\sum_{|\bfx|\leq X}\sum_{|\bfy|\leq Y}
e\bigl(\a F(\bfx;\bfy)\bigr).
\]
By orthogonality we have 
\[R_F(P)=\int_0^1S(\a)\,\d\a.\]

We now introduce the Hardy--Littlewood dissection. Given
$0<\theta\leq1$, define
\[
\grM(\theta):=\bigcup_{\substack{0\leq a\leq q\leq P^\theta\\(a,q)=1}}
\grM_{a,q}(\theta),
\]
where
\[
\grM_{a,q}(\theta):=\{\a\in[0,1):|q\a-a|\leq P^{\theta-1}\}.
\]
The corresponding minor arcs are
$\grm(\theta):=[0,1)\setminus\grM(\theta)$.

The main term is expressed through the following complete exponential
sums and oscillatory integrals:
\[
S(a,q):=\sum_{\bfx\in(\Z/q\Z)^{s_1}}
\sum_{\bfy\in(\Z/q\Z)^{s_2}}e\bigl(aF(\bfx;\bfy)/q\bigr),
\]
\[
I(\b;P):=\int_{|\bfgam|\leq X}\int_{|\bfrho|\leq Y}
e\bigl(\b F(\bfgam;\bfrho)\bigr)\,\d\bfgam\,\d\bfrho,
\]
and the truncated local functions
\[
\grS(Q):=\sum_{\substack{1\leq a\leq q\leq Q\\(a,q)=1}}
q^{-s_1-s_2}S(a,q),
\]
and
\[
J(Q;P):=\int_{|\b|\leq QP^{-1}}I(\b;P)\,\d\b.
\]

The following lemma isolates the two estimates required from the subsequent
minor and major arcs arguments.

\begin{lemma}\label{lemma: circlemethod dissection}
    Suppose that $0<\theta<\min\{w_1,w_2\}/(5d)$ and that
    \[\int_{\grm(\theta)}|S(\a)| \d \a = o(X^{s_1}Y^{s_2}P^{-1}).\]
    Then
    \[R_F(P) = \grS(P^{\theta}) J(P^{\theta};1) X^{s_1}Y^{s_2}P^{-1} + o(X^{s_1}Y^{s_2}P^{-1}).\]
    Moreover, if the singular series and singular integral converge
    absolutely to $\grS$ and $J$, respectively, then
    \[
        R_F(P)=X^{s_1}Y^{s_2}P^{-1}\bigl(\grS J+o(1)\bigr).
    \]
\end{lemma}
\begin{proof}
    We first enlarge the major arcs slightly by setting
    \[\grN(\theta) = \bigcup_{\substack{0 \le a \le q \le P^{\theta} \\ (a,q) = 1}}\grN_{a,q}(\theta),\]
where
\[\grN_{a,q}(\theta) = \{\a \in [0,1): |\a - a/q| \le P^{\theta - 1}\}.\]
The complementary arcs $\grn(\theta):=[0,1)\setminus\grN(\theta)$ satisfy
$\grn(\theta)\subset\grm(\theta)$. Hence,
\[
    R_F(P)=\int_{\grN(\theta)}S(\a)\,\d\a
    +o(X^{s_1}Y^{s_2}P^{-1}).
\]
Fix $\a\in\grN_{a,q}(\theta)$ and write $\b=\a-a/q$. Decomposing
$\bfx$ and $\bfy$ into residue classes modulo $q$, with
$\bfx=q\bfu+\bfr_1$ and $\bfy=q\bfv+\bfr_2$, gives
\begin{align*}
    S(\a) &= \sum_{\substack{\bfr_1 \in (\Z/q\Z)^{s_1} \\ \bfr_2 \in (\Z/q\Z)^{s_2}}} \sum_{\substack{\bfu,\bfv \\ |q \bfu + \bfr_1| \le X \\ |q \bfv + \bfr_2| \le Y}} e((a/q+\b)F(q \bfu + \bfr_1,q \bfv + \bfr_2)), \\
    &= \sum_{\substack{\bfr_1 \in (\Z/q\Z)^{s_1} \\ \bfr_2 \in (\Z/q\Z)^{s_2}}} e(aF(\bfr_1,\bfr_2)/q)\sum_{\substack{\bfu,\bfv \\ |q \bfu + \bfr_1| \le X \\ |q \bfv + \bfr_2| \le Y}} e(\b F(q \bfu + \bfr_1,q \bfv + \bfr_2)). \\
\end{align*}
Write
\[S_1 = \sum_{\substack{\bfu,\bfv \\ |q \bfu + \bfr_1| \le X \\ |q \bfv + \bfr_2| \le Y}} e(\b F(q \bfu + \bfr_1,q \bfv + \bfr_2)),\]
so that the mean value theorem, followed by a change of variables, yields uniformly
in $\bfr_1$ and $\bfr_2$
\[
S_1=q^{-(s_1+s_2)}I(\b;P)
+O\left(q^{-s_1-s_2}X^{s_1}Y^{s_2}P^{-\min\{w_1,w_2\}/d}
(q+qP|\b|)\right).
\]
Since $\a\in\grN_{a,q}(\theta)$, it follows that
\[S(\a) = q^{-(s_1+s_2)}S(a,q)I(\b;P) + O\left(qX^{s_1}Y^{s_2}P^{\theta-\min\{w_1,w_2\}/d} \right).\]
Summing over the enlarged major arcs, we obtain
\begin{align*}
    \int_{\grN(\theta)}S(\a) \d \a &= \sum_{\substack{1 \le a \le q \le P^{\theta} \\ (a,q) = 1}}\int_{|\a - a/q|\le P^{\theta - 1}}S(\a) \d \a +o(X^{s_1}Y^{s_2}P^{-1}), \\
    &= \grS(P^{\theta})J(P^{\theta};P) + O(X^{s_1}Y^{s_2}P^{5\theta-1-\min\{w_1,w_2\}/d }).
\end{align*}
The restriction on $\theta$ makes the displayed error
$o(X^{s_1}Y^{s_2}P^{-1})$. Finally, weighted homogeneity and the changes of
variables $\bfgam=X\bfu$, $\bfrho=Y\bfv$ give
\[
I(\b;P)=X^{s_1}Y^{s_2}I(P\b;1),
\qquad
J(P^\theta;P)=X^{s_1}Y^{s_2}P^{-1}J(P^\theta;1).
\]
If the singular series and singular integral converge absolutely, then
$\grS(P^\theta)=\grS+o(1)$ and $J(P^\theta;1)=J+o(1)$.  In particular,
the truncated factors are bounded, and therefore
\[
    \grS(P^\theta)J(P^\theta;1)=\grS J+o(1).
\]
Substitution into the first conclusion proves the final assertion.
\end{proof}

\section{Exponential sum estimates}\label{section4!}
Throughout this section, let
$F\in\Z[x_1,\ldots,x_{s_1},y_1,\ldots,y_{s_2}]$ be the weighted form
fixed in the introduction. Thus
\[
    w_1<w_2,
    \qquad
    \gcd(w_1,w_2)=1,
    \qquad
    d>w_1w_2.
\]
In particular, no divisibility condition on $d$ is imposed.
Let $\mathcal B_1\subset[-1,1]^{s_1}$ and
$\mathcal B_2\subset[-1,1]^{s_2}$ be axis-aligned boxes of side length at
most $1$, and let $X,Y$ and $P$ be arbitrary
large parameters satisfying
\[
    X\leq Y,
    \qquad
    \log X\asymp\log Y\asymp\log P.
\]
For $\a\in[0,1)$, define
\[
    S(\a)=
    \sum_{\bfx\in X\mathcal B_1}
    \sum_{\bfy\in Y\mathcal B_2}
    e\bigl(\a F(\bfx;\bfy)\bigr).
\]

In this section, we develop exponential sum estimates for weighted forms
by adapting the differencing methods of Schindler \cite{Schindler2014}. Our
first estimate isolates an arbitrary bihomogeneous component of the weighted
form using a differencing argument. We then combine this estimate with the
geometric arguments of Birch and Schindler to obtain a Weyl type lemma for
weighted forms whose conclusion depends on the individual singular loci of the
bihomogeneous components described below.

We use the tuple notation introduced in \eqref{eq:tuple notation}. Given a bihomogeneous form
$G^{(a,b)}(\bfx;\bfy)$ of bidegree $(a,b)$ and nonnegative integers
$j_1,j_2$, define its bihomogeneous differencing operator by
\begin{equation}\label{def of bihomogeneous differencing}
\begin{aligned}
G^{(a,b)}_{j_1,j_2}
    \bigl(\bfX^{(j_1)};\bfY^{(j_2)}\bigr)  =
\sum_{\substack{\bfeps\in\{0,1\}^{j_1}\\
                 \bfeta\in\{0,1\}^{j_2}}}
(-1)^{\|\bfeps\|_1+\|\bfeta\|_1}
G^{(a,b)}
\left(
    \sum_{i=1}^{j_1}\epsilon_i\bfx_i;
    \sum_{j=1}^{j_2}\eta_j\bfy_j
\right).
\end{aligned}
\end{equation}
Note that $G^{(a,b)}_{1,1}(\bfx;\bfy)=G^{(a,b)}(\bfx;\bfy)$ when $a,b>0$.
Whenever $a=0$ or $b=0$, we use the empty-tuple convention from Subsection
\ref{subsec:notational conventions}. In particular, for a form depending
only on the variables $\bfy$, the expression $G^{(0,b)}_{0,j_2}(\bfY^{(j_2)})$ is
defined by applying \eqref{def of bihomogeneous differencing} in the
$\bfy$-variables alone. The notation
$G^{(a,0)}_{j_1,0}(\bfX^{(j_1)})$ is interpreted analogously.

We record an explicit description of the indexing introduced in
\eqref{eqs:decomp}. Let $d_2(0)$ be the unique integer satisfying
\[
    0\leq d_2(0)<w_1,
    \qquad
    w_2d_2(0)\equiv d\pmod{w_1},
\]
and put
\[
    d_1(0)=\frac{d-w_2d_2(0)}{w_1},
    \qquad
    m=\left\lfloor\frac{d_1(0)}{w_2}\right\rfloor.
\]
Since $d_2(0)\leq w_1-1$, the assumption $d>w_1w_2$ gives
\[
    d-w_2d_2(0)>w_1w_2-w_2(w_1-1)=w_2,
\]
and hence $d_1(0)>0$.  Moreover, coprimality of $w_1$ and $w_2$ shows
that every solution of
\[
    w_1d_1+w_2d_2=d,
\]
has $d_2\equiv d_2(0)\pmod{w_1}$.  The nonnegative integer solutions
are therefore precisely
\[
    d_1(k)=d_1(0)-kw_2,
    \qquad
    d_2(k)=d_2(0)+kw_1,
    \qquad
    0\leq k\leq m.
\]
Accordingly, the bihomogeneous decomposition of $F$ is
\[
    F(\bfx;\bfy)
    =
    \sum_{0\leq k\leq m}
    F^{(d_1(k),d_2(k))}(\bfx;\bfy).
\]
After multiplying $F$ by
$\prod_{0\leq k\leq m}d_1(k)!\,d_2(k)!$, which changes neither its zero set
nor its singular loci, we may assume that the coefficients of each
bihomogeneous component are symmetric and integral, as in \cite[\S 2]{Schindler2014}.
In particular, the matrices $\bfM_k$ defined below are integral.
We also put $D(k):=d_1(k)+d_2(k)-2$. Recall from
\eqref{def:hat sigma} that
\[
\widehat{\sigma}(k):=
\begin{cases}
s_2+\dim V_1(F^{(d_1(k),0)}),&d_2(k)=0,\\[2mm]
\displaystyle\max_{i\in\{1,2\}}
\dim V_i(F^{(d_1(k),d_2(k))}),&d_1(k)d_2(k)>0,\\[2mm]
s_1+\dim V_2(F^{(0,d_2(k))}),&d_1(k)=0.
\end{cases}
\]
One may then extend the bihomogeneous difference operator to $F$ by applying it to each
of its bihomogeneous components. Thus, we write
\begin{equation}\label{def of weighted bihomogeneous differencing}
    F_{j_1,j_2}
    \bigl(\bfX^{(j_1)};\bfY^{(j_2)}\bigr)
    =
    \sum_{0\leq k\leq m}
    F^{(d_1(k),d_2(k))}_{j_1,j_2}
    \bigl(\bfX^{(j_1)};\bfY^{(j_2)}\bigr).
\end{equation}

Fix an index $k$ for which $d_1(k),d_2(k)>0$. By multilinearity of the
fully differenced form, there is a matrix
\[
    \bfM_k
    \bigl(
        \bfX^{(d_1(k)-1)};
        \bfY^{(d_2(k)-1)}
    \bigr)
    \in
    \Z[
        \bfX^{(d_1(k)-1)},
        \bfY^{(d_2(k)-1)}
    ]^{s_1\times s_2},
\]
such that
\begin{equation}\label{eqs: matrix def}
\begin{aligned}
F^{(d_1(k),d_2(k))}_{d_1(k),d_2(k)}
\bigl(
    \bfX^{(d_1(k))};
    \bfY^{(d_2(k))}
\bigr) =
d_1(k)!\,d_2(k)!\,
\bfx_{d_1(k)}^T
\bfM_k
\bigl(
    \bfX^{(d_1(k)-1)};
    \bfY^{(d_2(k)-1)}
\bigr)
\bfy_{d_2(k)}.
\end{aligned}
\end{equation}

We now define counting functions associated with this matrix. Let
$\mathcal N_{1,k}(\alpha;X,Y,Z)$ be the number of integral tuples
\[
    \bfX^{(d_1(k)-1)},
    \qquad
    \bfY^{(d_2(k)-1)},
    \qquad
    \bfy_{d_2(k)},
\]
satisfying
\[
    \left|\bfX^{(d_1(k)-1)}\right|\leq X,
    \qquad
    \left|\bfY^{(d_2(k)-1)}\right|\leq Y,
    \qquad
    \left|\bfy_{d_2(k)}\right|\leq Y,
\]
and
\[
\left\|
    \alpha d_1(k)!\,d_2(k)!\,
    \bfM_k
    \bigl(
        \bfX^{(d_1(k)-1)};
        \bfY^{(d_2(k)-1)}
    \bigr)
    \bfy_{d_2(k)}
\right\|
\leq Z^{-1}.
\]

Similarly, let $\mathcal N_{2,k}(\alpha;X,Y,Z)$ be the number of
integral tuples
\[
    \bfX^{(d_1(k)-1)},
    \qquad
    \bfY^{(d_2(k)-1)},
    \qquad
    \bfx_{d_1(k)},
\]
satisfying
\[
    \left|\bfX^{(d_1(k)-1)}\right|\leq X,
    \qquad
    \left|\bfY^{(d_2(k)-1)}\right|\leq Y,
    \qquad
    \left|\bfx_{d_1(k)}\right|\leq X,
\]
and
\[
\left\|
    \alpha d_1(k)!\,d_2(k)!\,
    \bfM_k
    \bigl(
        \bfX^{(d_1(k)-1)};
        \bfY^{(d_2(k)-1)}
    \bigr)^T
    \bfx_{d_1(k)}
\right\|
\leq Z^{-1}.
\]

\begin{lemma}\label{lemma: BihomWeyl full weighted form}
    Fix an integer $0\leq k_0\leq m$ for which
    $d_1(k_0),d_2(k_0)>0$,
    and put
    \[
        d_1=d_1(k_0),
        \qquad
        d_2=d_2(k_0),
        \qquad
        D=d_1+d_2-2.
    \]
    Let $H$ be a large parameter satisfying
    \[
        H\leq\min\{X,Y\},
        \qquad
        \log H\asymp\log P.
    \]
    Then, for every $Q\geq1$ and every $\eps>0$, either
    \[
        |S(\a)|\leq X^{s_1}Y^{s_2}Q^{-1},
    \]
    or, for at least one $i\in\{1,2\}$, one has
    \[
        \mathcal N_{i,k_0}
        \left(
            \a;H,H,
            X^{d_1}Y^{d_2}H^{-(D+1)}
        \right)
        \gg
        H^{s_1d_1+s_2d_2-s_i}P^{-\eps}Q^{-2^D}.
    \]
\end{lemma}
\begin{proof}
    We begin with the bihomogeneous decomposition
    \begin{equation}\label{eqs: decomposition into bihomo forms}
        F(\bfx;\bfy)
        =\sum_{0\leq k\leq m}
        F^{(d_1(k),d_2(k))}(\bfx;\bfy),
    \end{equation}
    where $d_1(k)=d_1(0)-kw_2$ and
    $d_2(k)=d_2(0)+kw_1$. We now define our exponential sum
    \[S(\a) = \sum_{\bfx \in X \mathcal{B}_1} \sum_{\bfy \in Y \mathcal{B}_2}e(\a F(\bfx;\bfy)),\]
    and
    \begin{equation}\label{def of Sy}
        S_{\bfy}(\a) = \sum_{\bfx \in X \mathcal{B}_1}e(\a F(\bfx;\bfy)).
    \end{equation}
    For now we shall write some notational shorthand in an attempt to keep our notation tidy. We will now write $B_1 = X \mathcal{B}_1$ and $B_2 = Y \mathcal{B}_2$. Fix $\bfy \in \Z^{s_2}$ and $\a \in [0,1)$, and use the notation
    \[d_1 = d_1(k_0) \ \text{and} \ d_2 = d_2(k_0).\]
    We first suppose that $d_1\geq2$.
    Since $d_1\geq2$, Lemma \ref{lemma:schindlerweyl} gives
    \begin{equation}\label{eqs:S_y bound}
        |S_{\bfy}(\a)|^{2^{d_1 - 1}} \ll X^{s_1(2^{d_1-1}-d_1)} \sum_{\substack{\bfx_1,\ldots,\bfx_{d_1-2} \in B_1^D \\ \bfx_{d_1-1} \in B_1(\bfx_1,\ldots,\bfx_{d_1-2})^D \\ \bfx_{d_1} \in B_1(\bfx_1,\ldots,\bfx_{d_1-1})}} e(\a(F_{d_1,1}(\bfX^{(d_1)};\bfy)-F_{d_1-1,1}(\bfX^{(d_1 -1)};\bfy))).
    \end{equation}
    By utilizing Lemma \ref{lemma:schmidtdifferenceidentify} along with our decomposition of $F$ \eqref{eqs: decomposition into bihomo forms} and noting that $w_2>w_1 \ge 1$, one sees that
    \[F_{d_1,1}(\bfX^{(d_1)};\bfy)-F_{d_1-1,1}(\bfX^{(d_1 -1)};\bfy) = \sum_{0 \le k \le k_0} \left(F^{(d_1(k),d_2(k))}_{d_1,1}(\bfX^{(d_1)};\bfy) - F_{d_1 -1,1}^{(d_1(k),d_2(k))}(\bfX^{(d_1 - 1)};\bfy)\right).\]
    In order to keep notation in check, we denote
    \[\calG(\bfX^{(d_1)};\bfy) = \sum_{0 \le k \le k_0} \left(F^{(d_1(k),d_2(k))}_{d_1,1}(\bfX^{(d_1)};\bfy) - F_{d_1 -1,1}^{(d_1(k),d_2(k))}(\bfX^{(d_1 - 1)};\bfy)\right).\]
    Thus, via the use of H\"{o}lder's inequality one may deduce the bound:
    \begin{align*}
        |S(\a)|^{2^{d_1 - 1}} &\ll Y^{s_2(2^{d_1 - 1}-1)} \sum_{\bfy \in B_2} |S_{\bfy}(\a)|^{2^{d_1 - 1}}, \\
        &\ll Y^{s_2(2^{d_1 - 1}-1)} X^{s_1(2^{d_1-1}-d_1)} \sum_{\substack{\bfy \in B_2 \\ \bfx_1,\ldots,\bfx_{d_1}}} e\left(\a \calG(\bfX^{(d_1)};\bfy) \right),
    \end{align*}
    where the sum over $\bfx_1,\ldots,\bfx_{d_1}$ is the same as in \eqref{eqs:S_y bound}. For fixed $\bfX^{(d_1)} \in \Z^{s_1 \times d_1}$ and $\a \in [0,1)$, when $d_2\geq2$ we apply Lemma \ref{lemma:schmidtweyl}. When $d_2=1$, the corresponding estimate follows directly from the triangle inequality, and no $\bfy$-differencing is required. With the empty-tuple convention from Subsection \ref{subsec:notational conventions}, both cases may be written as
    \[\Big|\sum_{\bfy \in B_2}e( \a \calG(\bfX^{(d_1)};\bfy)) \Big|^{2^{d_2 - 1}} \ll Y^{s_2(2^{d_2-1} -d_2)}  \sum_{\substack{\bfy_1,\ldots,\bfy_{d_2-1} \in B_2^D }} \Big|\sum_{\bfy_{d_2} \in B_2(\bfy_1,\ldots,\bfy_{d_2-1})} e( \a \calG_{d_2}(\bfX^{(d_1)};\bfY^{(d_2)})) \Big| .\]
    By Lemma \ref{lemma:schmidtdifferenceidentify}, we have the identity
    \begin{align*}
        \calG_{d_2} \bigl(\bfX^{(d_1)};\bfY^{(d_2)}\bigr) &= F^{(d_1,d_2)}_{d_1,d_2} \bigl(\bfX^{(d_1)};\bfY^{(d_2)}\bigr) -
    F^{(d_1,d_2)}_{d_1-1,d_2} \bigl(\bfX^{(d_1-1)};\bfY^{(d_2)}\bigr), \\
    &= d_1!d_2! \bfy_{d_2}^{T}
    \Bigl(
        \bfM_{k_0}\bigl(\bfX^{(d_1-1)};\bfY^{(d_2-1)}\bigr)^{T}\bfx_{d_1}
        +\bfL_{k_0}\bigl(\bfX^{(d_1-1)};\bfY^{(d_2-1)}\bigr)
    \Bigr).
    \end{align*}
    Here the matrix associated with the $k_0$-th bihomogeneous component is
    \[
        \bfM_{k_0}\bigl(\bfX^{(d_1-1)};\bfY^{(d_2-1)}\bigr),
    \]
    as defined in \eqref{eqs: matrix def}, while the vector of rational
    polynomials supplied by Lemma \ref{lemma:schmidtdifferenceidentify} is
    \[
        \bfL_{k_0}\bigl(\bfX^{(d_1-1)};\bfY^{(d_2-1)}\bigr)
        \in \Q[\bfX^{(d_1-1)},\bfY^{(d_2-1)}]^{s_2}.
    \]
    For notational convenience, throughout the remainder of the proof write
    \[
        \bfM_{k_0}
        =\bfM_{k_0}\bigl(\bfX^{(d_1-1)};\bfY^{(d_2-1)}\bigr)
        \qquad\text{and}\qquad
        \bfL_{k_0}=\bfL_{k_0}\bigl(\bfX^{(d_1-1)};\bfY^{(d_2-1)}\bigr).
    \]
    Setting $D = D(k_0) = d_1(k_0)+d_2(k_0)-2$, we deduce via H\"{o}lder's inequality the following bound:
    \begin{align*}
        |S(\a)|^{2^{D}} &\ll Y^{s_2(2^ {D}-2^{d_2-1})} X^{s_1(2^{D}-d_1 2^{d_2 -1})} \Big|\sum_{\substack{\bfy \in B_2 \\ \bfx_1,\ldots,\bfx_{d_1}}} e\left( \a \calG(\bfX^{(d_1)};\bfy) \right) \Big|^{2^{d_2-1}}, \\
        &\ll Y^{s_2(2^{D}-2^{d_2-1})} X^{s_1(2^{D}-d_1)} \sum_{\substack{\bfx_1,\ldots,\bfx_{d_1}}}\Big| \sum_{\bfy \in B_2} e\left( \a \calG(\bfX^{(d_1)};\bfy) \right) \Big|^{2^{d_2-1}}, \\
        &\ll Y^{s_2(2^{D}-d_2)} X^{s_1(2^{D}-d_1)} \sum_{\substack{ \bfx_1,\ldots,\bfx_{d_1-1} \\ \bfy_1,\ldots,\bfy_{d_2-1}}}\Xi, \\
    \end{align*}
    where
    \begin{equation*}
    \Xi = \sum_{\bfx_{d_1}\in
    B_1(\bfx_1,\ldots,\bfx_{d_1-1})}
    \Bigg|
    \sum_{\bfy_{d_2}\in
    B_2(\bfy_1,\ldots,\bfy_{d_2-1})} e\left(
        \a d_1!d_2!\,
        \bfy_{d_2}^{T}
        \bigl(
            \bfM_{k_0}^{T}\bfx_{d_1}+\bfL_{k_0}
        \bigr)
    \right)
    \Bigg|.
    \end{equation*}
    When $d_1=1$, the same estimate follows by the symmetric argument in
    which the $\bfy$-differencing is performed first.  More precisely,
    H\"older's inequality gives
    \[
        |S(\a)|^{2^{d_2-1}}
        \leq
        X^{s_1(2^{d_2-1}-1)}
        \sum_{\bfx\in B_1}
        \left|
            \sum_{\bfy\in B_2}e\bigl(\a F(\bfx;\bfy)\bigr)
        \right|^{2^{d_2-1}}.
    \]
    If $d_2\geq2$, we apply Lemma \ref{lemma:schindlerweyl} in the
    $\bfy$-variables.  If $d_2=1$, no differencing is required and the
    corresponding estimate follows from the triangle inequality.  Since
    $d_1(k)$ decreases and $d_2(k)$ increases with $k$, the resulting
    $d_2$-fold difference isolates the component indexed by $k_0$, up to a
    term independent of $\bfx$.  This term is absorbed into
    $\bfL_{k_0}$.  Leaving $\bfy_{d_2}$ as the final linear summation
    variable gives the same quantity $\Xi$ and the same bound
    \[
        |S(\a)|^{2^D}
        \ll
        Y^{s_2(2^D-d_2)}X^{s_1(2^D-d_1)}
        \sum_{\bfY^{(d_2-1)}}\Xi.
    \]
    Thus the following argument applies for every $d_1,d_2>0$.

    We now keep track explicitly of the constants arising from our differenced boxes. There are fixed constants $c_X,c_Y,c_0\geq 1$,
    depending only on $d_1,d_2,\mathcal B_1$ and $\mathcal B_2$, such that
    Lemma \ref{lemma:bilinearbound} gives
    \[
        |S(\a)|^{2^{D}}
        \ll
        Y^{s_2(2^D-d_2+1)+\eps}X^{s_1(2^D-d_1)}
        \widetilde{\mathcal N}_{2,k_0}(\a;X,Y),
    \]
    where $\widetilde{\mathcal N}_{2,k_0}(\a;X,Y)$ counts the integral tuples
    $\bfX^{(d_1-1)},\bfY^{(d_2-1)},\bfx_{d_1}$ satisfying
    \[
        |\bfX^{(d_1-1)}|\leq c_XX,
        \qquad
        |\bfY^{(d_2-1)}|\leq c_YY,
        \qquad
        |\bfx_{d_1}|\leq c_XX,
    \]
    and
    \[
        \bigl\|
        \a d_1!d_2!\,\bfM_{k_0}^{T}\bfx_{d_1}
        \bigr\|
        \leq c_0Y^{-1}.
    \]

    Put
    \[
        C=\max\{c_X,c_Y,c_0,1\},
        \qquad Z=CY,
        \qquad L=c_0C,
    \]
    and let $\mathcal M_{2,k_0}(\a;X,Y;Z,L)$ denote the number of integral
    tuples $\bfX^{(d_1-1)},\bfY^{(d_2-1)},\bfx_{d_1}$ for which
    \[
        |\bfX^{(d_1-1)}|\leq CX,
        \qquad
        |\bfY^{(d_2-1)}|\leq Z,
        \qquad
        |\bfx_{d_1}|\leq CX,
    \]
    and
    \[
        \bigl\|
        \a d_1!d_2!\,\bfM_{k_0}^{T}\bfx_{d_1}
        \bigr\|
        \leq LZ^{-1}.
    \]
    Since $C\geq c_X,c_Y$, we have
    \[
        c_XX\leq CX,
        \qquad
        c_YY\leq CY=Z.
    \]
    Moreover, the definitions of $Z$ and $L$ give the exact identity
    \[
        c_0Y^{-1}=LZ^{-1}.
    \]
    Every tuple counted by
    $\widetilde{\mathcal N}_{2,k_0}(\a;X,Y)$ is therefore counted by
    $\mathcal M_{2,k_0}(\a;X,Y;Z,L)$, whence
    \[
        \widetilde{\mathcal N}_{2,k_0}(\a;X,Y)
        \leq
        \mathcal M_{2,k_0}(\a;X,Y;Z,L).
    \]
    It follows that
    \[
        |S(\a)|^{2^{D}}
        \ll
        Y^{s_2(2^D-d_2+1)+\eps}X^{s_1(2^D-d_1)}
        \mathcal M_{2,k_0}(\a;X,Y;Z,L).
    \]

    It follows that if
    \[
        X^{s_1}Y^{s_2}Q^{-1}<|S(\a)|,
    \]
    then
    \[
        \mathcal M_{2,k_0}(\a;X,Y;Z,L)
        \gg
        X^{s_1d_1}Y^{s_2(d_2-1)-\eps}Q^{-2^D}.
    \]

    As in Schindler \cite[\S 4]{Schindler2014}, we apply Lemma
    \ref{lemma:lattice point comparison} in each differencing variable in order to
    bound $\mathcal M_{2,k_0}(\a;X,Y;Z,L)$ in terms of
    the counting functions $\mathcal N_{1,k_0}$ and $\mathcal N_{2,k_0}$
    with their first two arguments both equal to $H$. For each fixed choice 
    of the remaining differencing variables, we apply Lemma \ref{lemma:lattice point comparison} 
    to the selected variable and then sum over all admissible choices of the remaining variables. 
    Specifically, we apply the lemma first to $\bfx_{d_1}$, then successively to
    \[
        \bfx_1,\ldots,\bfx_{d_1-1},
    \]
    and finally to
    \[
        \bfy_1,\ldots,\bfy_{d_2-1}.
    \]

    Initially, every $\bfx$-variable has range $CX$, every $\bfy$-variable
    has range $CY$, and the upper bound in the simultaneous inequalities
    defining $\mathcal M_{2,k_0}(\a;X,Y;Z,L)$ is
    \[
        LZ^{-1}=c_0Y^{-1}.
    \]
    After $r$ applications to $\bfx$-variables and $t$ applications to
    $\bfy$-variables, the upper bound in these simultaneous inequalities is
    \[
        \frac{c_0}{Y}
        \left(\frac{H}{CX}\right)^r
        \left(\frac{H}{CY}\right)^t.
    \]
    Indeed, an application to a $\bfx$-variable replaces its range $CX$ by
    $H$ and multiplies this upper bound by $H/(CX)$; an application to a
    $\bfy$-variable has the analogous effect with $CY$ in place of $CX$.

    We verify that Lemma \ref{lemma:lattice point comparison} is applicable at every
    application. Before the $(r+1)$-st application to a $\bfx$-variable, where
    $0\leq r<d_1$, choose $a,Z_1,Z_2>0$ so that
    \[
        aZ_2=CX,
        \qquad
        a^{-1}Z_2=\frac{c_0}{Y}
        \left(\frac{H}{CX}\right)^r,
        \qquad
        aZ_1=H.
    \]
    Thus
    \[
        a^{-1}Z_1=\frac{c_0}{Y}
        \left(\frac{H}{CX}\right)^{r+1},
        \qquad
        \frac{Z_1}{Z_2}=\frac{H}{CX},
    \]
    and
    \[
        Z_2^2
        =CX\frac{c_0}{Y}
        \left(\frac{H}{CX}\right)^r
        \leq Cc_0,
    \]
    since $X\leq Y$ and $H\leq X$. After all the $\bfx$-variables have been
    shortened, and before the $(t+1)$-st application to a $\bfy$-variable,
    where $0\leq t<d_2-1$, choose $a,Z_1,Z_2>0$ so that
    \[
        aZ_2=CY,
        \qquad
        a^{-1}Z_2=\frac{c_0}{Y}
        \left(\frac{H}{CX}\right)^{d_1}
        \left(\frac{H}{CY}\right)^t,
        \qquad
        aZ_1=H.
    \]
    Then
    \[
        a^{-1}Z_1=\frac{c_0}{Y}
        \left(\frac{H}{CX}\right)^{d_1}
        \left(\frac{H}{CY}\right)^{t+1},
        \qquad
        \frac{Z_1}{Z_2}=\frac{H}{CY},
    \]
    and
    \[
        Z_2^2
        =CY\frac{c_0}{Y}
        \left(\frac{H}{CX}\right)^{d_1}
        \left(\frac{H}{CY}\right)^t
        \leq Cc_0.
    \]
    Moreover, $H\leq CX,CY$, so in each application $0<Z_1\leq Z_2$.
    Thus Lemma \ref{lemma:lattice point comparison} applies throughout with a fixed
    upper bound for $Z_2$.

    Each application of Lemma \ref{lemma:lattice point comparison} gives a bound in
    terms of the maximum of two quantities, corresponding to the original and
    transposed systems of multilinear forms. We bound this maximum by their
    sum. After all $D+1$ applications, we therefore obtain a sum of at most
    $2^{D+1}$ counting functions. The fully differenced form
    $F^{(d_1,d_2)}_{d_1,d_2}$ is symmetric in
    $\bfx_1,\ldots,\bfx_{d_1}$ and, separately, in
    $\bfy_1,\ldots,\bfy_{d_2}$. Hence, after permuting the $\bfx$-variables
    among themselves and the $\bfy$-variables among themselves, each of these
    counting functions has one of the two forms defining
    $\mathcal N_{1,k_0}$ and $\mathcal N_{2,k_0}$. Since all the
    differencing variables have magnitude at most $H$, these permutations do
    not alter the bounds on the variables, and the first two arguments of the
    resulting counting functions are both $H$.
    
    Forcing all differencing variables to have magnitude at most $H$ is essential for this 
    identification, as otherwise we would need to sum a large number of counting functions 
    which are not directly comparable to $\calN_{1,k_0}$ or $\calN_{2,k_0}$.

    After all $d_1+d_2-1=D+1$ applications of Lemma \ref{lemma:lattice point comparison}, we deduce
    \[
        \frac{\mathcal M_{2,k_0}(\a;X,Y;Z,L)}
        {X^{s_1d_1}Y^{s_2(d_2-1)}}
        \ll
        \sum_{i=1}^{2}
        \frac{
            \mathcal N_{i,k_0}
            \left(
                \a;H,H,
                X^{d_1}Y^{d_2}H^{-(D+1)}
            \right)
        }{H^{s_1d_1+s_2d_2-s_i}}.
    \]
    Combining this estimate with the preceding lower bound, and using
    $\log X\asymp\log Y\asymp\log P$ to absorb $Y^{-\eps}$ into
    $P^{-\eps}$ after relabelling $\eps$, shows that for at least one
    $i\in\{1,2\}$,
    \[
        \mathcal N_{i,k_0}
        \left(
            \a;H,H,
            X^{d_1}Y^{d_2}H^{-(D+1)}
        \right)
        \gg
        H^{s_1d_1+s_2d_2-s_i}P^{-\eps}Q^{-2^D}.
    \]
    This is the desired alternative.

\end{proof}

\begin{lemma}\label{lemma:hat sigma weighted Weyl alternative}
    Define
    \[
        \eta_*
        =\min_{0\leq k\leq m}
        \begin{cases}
        \bigl(D(k)+1\bigr)\dfrac{\log Y}{\log P},
            & d_1(k)=0,\\[3mm]
        \bigl(D(k)+1\bigr)\dfrac{\log X}{\log P},
            & d_1(k)>0.
        \end{cases}
    \]
    Let $\kappa>0$ and $0<\eta\leq\eta_*$. Then at least one of the
    following holds.
    \begin{itemize}
        \item One has
        \[
            |S(\a)|\ll X^{s_1}Y^{s_2}P^{-\kappa+\eps}.
        \]
        \item For some $0\leq k\leq m$, there exist coprime
        integers $a,q$ such that
        \[
            1\leq q\leq P^\eta,
            \qquad
            |q\a-a|
            \leq
            X^{-d_1(k)}Y^{-d_2(k)}P^\eta.
        \]
        \item For every $0\leq k\leq m$ one has
        \[
            s_1+s_2-\widehat{\sigma}(k)
            \leq
            \varkappa(k)\frac{\kappa}{\eta},
        \]
        where
        \[
            \varkappa(k)=
            \begin{cases}
                2^{D(k)+1}(D(k)+1),
                    & d_1(k)d_2(k)=0,\\
                2^{D(k)}(D(k)+1),
                    & d_1(k)d_2(k)>0.
            \end{cases}
        \]
    \end{itemize}
\end{lemma}

\begin{proof}
    Fix an index $0\leq k\leq m$ for which
    $d_1(k),d_2(k)>0$, and abbreviate
    \[
        d_1=d_1(k),
        \qquad
        d_2=d_2(k),
        \qquad
        D=d_1+d_2-2,
        \qquad
        r=D+1.
    \]
    Let $C_k\geq 1$ be a fixed constant such that every coordinate of
    either multilinear vector occurring below has absolute value at most
    $C_kH^r$ whenever all its arguments lie in $[-H,H]$. Apply Lemma
    \ref{lemma: BihomWeyl full weighted form} with
    \[
        H=C_k^{-1/r}P^{\eta/r},
        \qquad
        Q=P^\kappa.
    \]
    In order to apply Lemma
    \ref{lemma: BihomWeyl full weighted form}, we require $H\leq X$.
    This follows from the definition of $\eta_*$, which gives
    \[
        \eta\leq r\frac{\log X}{\log P}.
    \]
    If the first
    alternative of Lemma \ref{lemma: BihomWeyl full weighted form} holds,
    then the first conclusion above follows. Otherwise, for at least one
    $i\in\{1,2\}$,
    \begin{equation}\label{eq:hat sigma large count}
        \mathcal N_{i,k}
        \left(\a;H,H,X^{d_1}Y^{d_2}H^{-r}\right)
        \gg
        H^{s_1d_1+s_2d_2-s_i}
        P^{-\eps-2^D\kappa}.
    \end{equation}

    The relevant integral vectors are
    \[
        d_1!d_2!\,\bfM_k\bfy_{d_2}
        \qquad\text{or}\qquad
        d_1!d_2!\,\bfM_k^T\bfx_{d_1}.
    \]
    Suppose first that some tuple counted on the left-hand side of
    \eqref{eq:hat sigma large count} makes the relevant vector nonzero.
    Choose a nonzero coordinate $q_0$ of this vector. By
    the choice of $C_k$ and the bounds on the differencing variables,
    \[
        1\leq |q_0|\leq C_kH^r=P^\eta.
    \]
    The defining inequalities for $\mathcal N_{i,k}$ provide an integer
    $a_0$ for which
    \[
        |q_0\a-a_0|
        \leq
        X^{-d_1}Y^{-d_2}H^r
        \leq X^{-d_1}Y^{-d_2}P^\eta.
    \]
    After changing signs if necessary and dividing $a_0,q_0$ by their
    greatest common divisor, we obtain the second conclusion. The resulting
    bounds have the same constant-free form as the rational-approximation
    alternative in \cite[Lemma 2.5]{Birch1962}; the choice of $C_k$ above
    makes the absorption of the fixed coefficient constants explicit here.

    It remains to consider the case in which every tuple counted in
    \eqref{eq:hat sigma large count} makes the relevant multilinear
    vector vanish. Applying Lemma
    \ref{lemma: BihomWeyl full weighted form} with $Q=P^\kappa$ as above,
    and following Schindler's identification of these multilinear zero
    sets with the corresponding singular loci
    \cite[Lemmas 4.2 and 4.3]{Schindler2014}, we have
    \[
        \mathcal N_{i,k}
        \left(\a;H,H,X^{d_1}Y^{d_2}H^{-r}\right)
        \ll
        H^{s_1d_1+s_2d_2-s_i-\codim{V_i(F^{(d_1,d_2)})}+\eps}.
    \]
    Comparing this with \eqref{eq:hat sigma large count}, using
    $H\asymp P^{\eta/r}$, and then letting $\eps$ tend to zero gives
    \[
        s_1+s_2-\widehat{\sigma}(k)
        \leq
        2^D r\frac{\kappa}{\eta}
        =\varkappa(k)\frac{\kappa}{\eta}.
    \]

    We apply this argument separately to every index for which
    $d_1(k),d_2(k)>0$. Hence, unless the first or second conclusion has
    already occurred, the stated singular-locus inequality holds for every
    such index. For an index for which $d_1(k)d_2(k)=0$, the corresponding
    assertion follows from Birch \cite[Lemma 2.5]{Birch1962}.  The
    definition of $\eta_*$ ensures that the differencing parameter is at
    most $X$ for a pure $\bfx$-component and at most $Y$ for a pure
    $\bfy$-component.
    Indeed, fully differencing in the $\bfx$-variables causes the
    dependence on the $\bfy$-variables to vanish, while fully
    differencing in the $\bfy$-variables causes the dependence on the
    $\bfx$-variables to vanish. Hence
    \[
        \varkappa(k)
        =2^{d_1(k)+d_2(k)-1}
        \bigl(d_1(k)+d_2(k)-1\bigr)
        \qquad\bigl(d_1(k)d_2(k)=0\bigr),
    \]
    which completes the proof.
\end{proof}

The restriction $\eta\leq\eta_*$ is analogous to the restriction on the
common differencing scale in Schindler's argument.  Indeed, Schindler writes
$P_1=P_2^b$, where $P_2\leq P_1$, and introduces a common scale
$P^\theta=P_1^{\theta_1}=P_2^{\theta_2}$, with
$P=P_1^{d_1}P_2^{d_2}$.  The requirement $\theta_2\leq1$, which ensures
that the common scale does not exceed the shorter side $P_2$, is equivalent
to
\[
    \theta\leq\frac{1}{bd_1+d_2},
\]
see \cite[Lemma 4.3]{Schindler2014}.  In our setting, the corresponding
common scale for a component for which $d_1(k),d_2(k)>0$ is
\[
    H\asymp P^{\eta/(d_1(k)+d_2(k)-1)},
\]
and requiring $H\leq X$ gives
\[
    \eta\leq
    \bigl(d_1(k)+d_2(k)-1\bigr)\frac{\log X}{\log P}.
\]
The same restriction applies to a pure $\bfx$-component.  For a pure
$\bfy$-component, the differencing scale is instead required to be at most
$Y$, giving
\[
    \eta\leq
    \bigl(d_2(k)-1\bigr)\frac{\log Y}{\log P}.
\]
Taking the minimum of these quantities over all the bihomogeneous
components gives the definition of $\eta_*$ above.

\begin{lemma}\label{lemma:hat sigma pointwise minor arc bound}
    Suppose that
    \[
        X=P^{w_1/d},
        \qquad
        Y=P^{w_2/d}.
    \]
    Put
    \[
        \eta_*
        :=\min_{0\leq k\leq m}
        \begin{cases}
        \dfrac{w_2}{d}\bigl(D(k)+1\bigr),&d_1(k)=0,\\[3mm]
        \dfrac{w_1}{d}\bigl(D(k)+1\bigr),&d_1(k)>0,
        \end{cases}
    \]
    and set
    \[
        K:=K(F)=\max_{0\leq k\leq m}
        \frac{s_1+s_2-\widehat{\sigma}(k)}{\varkappa(k)}.
    \]
    If $K>0$, then, for every $0<\eta\leq\eta_*$, every $\eps>0$, and
    every $\a$ in the minor arcs $\grm(\eta)$, one has
    \begin{equation}\label{eqs: pointwise minor arc bound}
        |S(\a)|
        \ll X^{s_1}Y^{s_2}P^{-K\eta+\eps},
    \end{equation}
    where $\grm(\eta)=[0,1)\setminus\grM(\eta)$ and $\grM(\eta)$ is the
    set of major arcs defined in Section \ref{sec:circle method}.
\end{lemma}

\begin{proof}
    Weighted homogeneity gives
    \[
        w_1d_1(k)+w_2d_2(k)=d,
    \]
    and hence
    \[
        X^{-d_1(k)}Y^{-d_2(k)}=P^{-1},
    \]
    for every $0\leq k\leq m$. Choose
    $0<\eps<K$ sufficiently small.
    Apply Lemma \ref{lemma:hat sigma weighted Weyl alternative} with
    $\kappa=(K-\eps)\eta$. If $\a\in\grm(\eta)$, then the second
    alternative of Lemma
    \ref{lemma:hat sigma weighted Weyl alternative} does not hold. The third
    alternative would imply
    \[
        \frac{s_1+s_2-\widehat{\sigma}(k)}{\varkappa(k)}
        \leq K-\eps
        \qquad
        \left(0\leq k\leq m\right),
    \]
    contradicting the definition of $K$. The first alternative therefore
    holds and gives an exponent of $-K\eta+\eps\eta+\eps$.
    Relabelling $\eps\eta+\eps$ as $\eps$ gives
    \eqref{eqs: pointwise minor arc bound}.
\end{proof}

\begin{lemma}\label{lemma: minor arc bound}
    Suppose that $K > \max\{ 2, 1/\eta_*\}$. Then there exist $\delta>0$ and
    \[
        0<\vartheta_0<\frac{w_1}{5d},
    \]
    such that
    \[\int_{\a \in \grm(\vartheta_0)} |S(\a)| \d \a \ll X^{s_1}Y^{s_2}P^{-1-\delta}.\]
\end{lemma}
\begin{proof}
    Choose $\vartheta_0>0$ sufficiently small that
    \[
        \vartheta_0<\min\left\{\frac{w_1}{5d},\frac1K\right\},
        \qquad
        (K-2)\vartheta_0<K\eta_*-1,
    \]
    and then choose $\delta>0$ so that
    \[
        2\delta<(K-2)\vartheta_0.
    \]
    Put
    \[
        a=\left(1-\frac{2}{K}\right)\vartheta_0
        -\frac{2\delta}{K}>0,
        \qquad
        \vartheta_t=\vartheta_0+at.
    \]
    Let $T$ be the least positive integer for which
    \[
        K\vartheta_T\geq1+2\delta.
    \]
    The minimality of $T$ gives
    \[
        \vartheta_T<\frac{1+2\delta}{K}+a
        =\frac1K+\left(1-\frac2K\right)\vartheta_0
        <\eta_*,
    \]
    where the final inequality follows from the choice of $\vartheta_0$.
    Thus, the result of Lemma
    \ref{lemma:hat sigma pointwise minor arc bound} is available at every
    $\vartheta_t$ with $0\leq t\leq T$.

    By the definition of the major arcs, we have
    \[
        \grm(\vartheta_0)
        =\grm(\vartheta_T)\sqcup
        \bigsqcup_{1\leq t\leq T}
        \left(\grM(\vartheta_t)\setminus
        \grM(\vartheta_{t-1})\right).
    \]
    Moreover,
    \[
        \operatorname{meas}(\grM(\vartheta))
        \ll P^{2\vartheta-1}.
    \]
    The set
    $\grM(\vartheta_t)\setminus\grM(\vartheta_{t-1})$ is contained in
    $\grm(\vartheta_{t-1})$. Hence, by
    \eqref{eqs: pointwise minor arc bound},
    \begin{align*}
        \int_{\grm(\vartheta_0)}|S(\a)|\,\d\a
        &\ll X^{s_1}Y^{s_2}P^\eps
        \left(
            P^{-K\vartheta_T}
            +P^{-1}\sum_{1\leq t\leq T}
            P^{2\vartheta_t-K\vartheta_{t-1}}
        \right).
    \end{align*}
    Since $K\vartheta_T\geq1+2\delta$, the first term in parentheses is
    at most $P^{-1-2\delta}$. Furthermore, $K>2$ and
    \[
        2\vartheta_t-K\vartheta_{t-1}
        =(2-K)\vartheta_{t-1}+2a
        \leq(2-K)\vartheta_0+Ka
        =-2\delta.
    \]
    Since $T\ll1$, relabelling $\eps$ now gives the required estimate.
\end{proof}

\section{The singular series and singular integral}\label{sec:major arcs}
The singular series and singular integral are common to all three of our main
theorems. In this section, we study the absolute convergence of
\[
    \grS=\sum_{q=1}^{\infty}
    \sum_{\substack{1\leq a\leq q\\(a,q)=1}}
    q^{-s_1-s_2}S(a,q)
    \qquad\text{and}\qquad
    J=\int_{-\infty}^{\infty}I(\b;1)\,\d\b,
\]
as well as their relation to nonsingular local solutions. The arguments below
use the estimates from Section \ref{section4!} and therefore give the hypotheses needed
for Theorem \ref{thm:V1-V2-criterion}. The same local factors will be used
for Theorems \ref{thm:weighted-singular-locus} and
\ref{thm:component-singular-locus}, with the corresponding exponential
sum estimates substituted into the convergence arguments.

\begin{lemma}\label{lemma: complete singular series}
    Under the assumption that $K > 2$, the complete singular series $\grS$ is absolutely convergent and one has
    \[|\grS - \grS(Q)| \ll Q^{1-K/2}.\]
\end{lemma}
\begin{proof}
    Fix $0<\tau<1$.  We apply Lemma
    \ref{lemma:hat sigma weighted Weyl alternative} with
    \[
        X=Y=P=q,
        \qquad
        \a=\frac aq,
        \qquad
        \eta=1-\tau,
    \]
    where $(a,q)=1$.  With $X=Y=P=q$, the definition of $\eta_*$ in
    Lemma \ref{lemma:hat sigma weighted Weyl alternative} gives
    \[
        \eta_*
        =\min_{0\leq k\leq m}
        \bigl(d_1(k)+d_2(k)-1\bigr).
    \]
    Every admissible bidegree satisfies $d_1(k)+d_2(k)\geq2$.  Indeed,
    an admissible bidegree of ordinary degree $1$ would imply that
    $d\in\{w_1,w_2\}$, contrary to $d>w_1w_2$.  Therefore
    $\eta_*\geq1$, and the choice $\eta=1-\tau$ is admissible.

    Choose $0<\rho<K$ and set $\kappa=(K-\rho)\eta$.  The third
    alternative of Lemma \ref{lemma:hat sigma weighted Weyl alternative} would give
    \[
        \frac{s_1+s_2-\widehat{\sigma}(k)}{\varkappa(k)}\leq K-\rho
        \qquad
        \left(0\leq k\leq m\right),
    \]
    contrary to the definition of $K$.  Suppose that the second
    alternative holds.  Then for some $k$ there are coprime integers
    $a',q'$ such that
    \[
        1\leq q'\leq q^{1-\tau},
        \qquad
        \left|q'\frac aq-a'\right|
        \leq q^{-d_1(k)-d_2(k)+1-\tau}.
    \]
    Multiplication by $q$ gives
    \[
        |q'a-a'q|
        \leq q^{2-d_1(k)-d_2(k)-\tau}<1.
    \]
    The integer on the left must therefore vanish.  Since $(a,q)=1$,
    this implies $q\mid q'$, contradicting $q'<q$ when $q>1$.
    Hence the first alternative holds.  By first choosing $\rho$ and
    $\tau$ sufficiently small and then relabelling $\eps$, we obtain
    \[
        S(a,q)\ll q^{s_1+s_2-K+\eps}.
    \]
    It follows that
    \[
        |\grS-\grS(Q)|
        \leq
        \sum_{q>Q}\sum_{\substack{1\leq a\leq q\\(a,q)=1}}
        q^{-s_1-s_2}|S(a,q)|
        \ll Q^{2-K+\eps}
        \ll Q^{1-K/2},
    \]
    where the last estimate follows by taking
    $0<\eps<(K-2)/2$.
\end{proof}

Whenever $\grS$ converges absolutely, let
\[A(q) = q^{-s_1-s_2}\sum_{\substack{1 \le a \le q \\ (a,q) = 1}}S(a,q),\]
then the general arguments of \cite[Section 17]{Davenport2005} give
\[\grS = \prod_{p}\sigma_p \ \text{where} \    \sigma_p = \sum_{h=0}^{\infty}A(p^h).\]
One may then show, via standard techniques, that
\[\sigma_p = \lim_{h \to \infty}p^{h(1-s_1-s_2)} \# \{(\bfx;\bfy) \in (\Z/p^h\Z)^{s_1+s_2}: F(\bfx;\bfy) = 0 \}.\]
The standard Hensel-lifting argument, as described in
\cite[Section 17]{Davenport2005}, then shows that $\sigma_p>0$ whenever
there exists a nonsingular $p$-adic solution to the equation
\[F(\bfx;\bfy) = 0.\]

\begin{lemma}\label{lemma: complete singular integral}
    Under the assumption that $K > 2$, the complete singular integral $J$ is absolutely convergent and one has
    \[|J - J(Q;1)| \ll Q^{-1}.\]
\end{lemma}
\begin{proof}
    Since $I(-\b;1)=\overline{I(\b;1)}$, it is enough to consider
    $\b>1$.  Fix $0<\theta<1/(5d)$ and choose $P>1$ so that
    $\b=P^\theta$.  Then
    \[
        P^{-1}\b=P^{\theta-1}
        \in\grM_{0,1}(\theta)=\grN_{0,1}(\theta).
    \]
    The major-arc approximation established in the proof of Lemma
    \ref{lemma: circlemethod dissection} therefore gives
    \begin{equation}\label{eqs: bound on edge of major arcs}
        S(P^{-1}\b) = X^{s_1}Y^{s_2}I(\b;1) + O\left(X^{s_1}Y^{s_2}P^{\theta-\min\{w_1,w_2\}/d} \right).
    \end{equation}
    Choose $\eps>0$ sufficiently small and put
    $\theta'=\theta-\eps/K>0$.  By the definition of the major arcs,
    $\grM(\theta')\subset\grM(\theta)$.  Since the arcs comprising
    $\grM(\theta)$ are disjoint, membership of $P^{-1}\b$ in
    $\grM(\theta')$ would force it to lie in
    $\grM_{0,1}(\theta')$.  This is impossible because
    \[
        P^{-1}\b=P^{\theta-1}>P^{\theta'-1}.
    \]
    Hence $P^{-1}\b\in\grm(\theta')$.  Applying Lemma
    \ref{lemma:hat sigma pointwise minor arc bound} at $\theta'$ gives
    \begin{equation}\label{eqs: bound on edge of minor arcs}
        |S(P^{-1}\b)|
        \ll X^{s_1}Y^{s_2}P^{-K\theta'+\eps}
        =X^{s_1}Y^{s_2}P^{-K\theta+2\eps}.
    \end{equation}
    Since every admissible bidegree satisfies
    $d_1(k)+d_2(k)\geq2$, the definition of $\eta_*$ in Lemma
    \ref{lemma:hat sigma pointwise minor arc bound} gives
    $\eta_*\geq w_1/d$.  Hence
    $\theta'<\theta<1/(5d)<\eta_*$, so Lemma
    \ref{lemma:hat sigma pointwise minor arc bound} is applicable here.
    Combining \eqref{eqs: bound on edge of major arcs} and
    \eqref{eqs: bound on edge of minor arcs}, we obtain
    \[
        I(\b;1)
        \ll P^{\theta-w_1/d}+P^{-K\theta+2\eps}
        =\b^{1-w_1/(d\theta)}
        +\b^{-K+2\eps/\theta}.
    \]
    The choice $\theta<1/(5d)$ gives
    $1-w_1/(d\theta)<1-5w_1\leq-4$.  Since $K>2$, we may take
    $\eps$ sufficiently small that $-K+2\eps/\theta<-2$.  It follows
    that $I(\b;1)\ll\b^{-2}$ for $\b>1$.  Therefore
    \[
        |J-J(Q;1)|
        \ll\int_{|\b|>Q}|I(\b;1)|\,\d\b
        \ll Q^{-1}.
    \]
\end{proof}

Whenever the singular integral is absolutely convergent, the analysis of Br\"{u}dern and
Wooley in the final section of \cite{Wooley2019} shows that $J$ is nonzero
when there exists a nonsingular real solution
$(\bfx_0,\bfy_0)\in[-1,1]^{s_1+s_2}$ to the equation
\[F(\bfx;\bfy) = 0.\]

\section{Deduction of the main theorems}

For $0\leq k\leq m$, put
\[
    r_k:=d_1(k)+d_2(k)-1,
\]
and let
\[
    \ell:=\max\{0\leq k\leq m:d_1(k)>0\}.
\]
The assumption $d>w_1w_2$ ensures that
$d_1(\ell),d_2(\ell)>0$. Define
\[
    \mathcal C(d;w_1,w_2)
    :=\max\left\{2,\frac{d}{w_1r_\ell}\right\}.
\]

We first relate this constant to the quantity $\eta_*$ in Lemma
\ref{lemma:hat sigma pointwise minor arc bound}. Since
\[
    r_{k+1}=r_k-(w_2-w_1),
\]
the sequence $r_k$ is strictly decreasing. Thus the minimum of
$w_1r_k/d$ over the indices for which $d_1(k)>0$ is attained at
$k=\ell$. If $\ell<m$, then $m=\ell+1$, $d_1(m)=0$, and
\[
    d_1(\ell)=w_2,
    \qquad
    d_2(m)=d_2(\ell)+w_1.
\]
Consequently,
\[
    w_2r_m-w_1r_\ell
    =(w_2-w_1)\bigl(d_2(\ell)-1\bigr)\geq0.
\]
It follows from the definition in Lemma
\ref{lemma:hat sigma pointwise minor arc bound} that
\[
    \eta_*=\frac{w_1r_\ell}{d},
    \qquad
    \mathcal C(d;w_1,w_2)=\max\{2,1/\eta_*\}.
\]

We shall also use a simpler upper bound for this constant. Write
\[
    a=d_1(\ell),
    \qquad
    b=d_2(\ell).
\]
Then $a,b\geq1$, $d=w_1a+w_2b$, and $r_\ell=a+b-1$. Since
\[
    (w_1+w_2)(a+b-1)-(w_1a+w_2b)
    =w_2(a-1)+w_1(b-1)\geq0,
\]
we have
\begin{equation}\label{eq:C simple upper bound}
    \mathcal C(d;w_1,w_2)
    \leq 1+\frac{w_2}{w_1}.
\end{equation}
Here we have also used $w_2>w_1$, which gives
$1+w_2/w_1>2$.

\subsection{Proof of Theorem \ref{thm:V1-V2-criterion}}

We first state the sharper form supplied directly by the exponential-sum
estimates.

\begin{theorem}\label{thm:V1-V2-criterion-precise}
Suppose that $d>w_1w_2$ and that, for some $0\leq k\leq m$, one has
\[
    s_1+s_2-\widehat{\sigma}(k)
    >\mathcal C(d;w_1,w_2)\varkappa(k),
\]
where
\[
    \varkappa(k)=
    \begin{cases}
        2^{r_k}r_k,&d_1(k)d_2(k)=0,\\[2mm]
        2^{r_k-1}r_k,&d_1(k)d_2(k)>0.
    \end{cases}
\]
Then, as $P\to\infty$, one has
\[
    R_F(P)
    =\mathfrak S J\,
    P^{(s_1w_1+s_2w_2)/d-1}
    +o\left(P^{(s_1w_1+s_2w_2)/d-1}\right),
\]
where $\mathfrak S$ and $J$ are the singular series and singular integral
associated with $F$. Moreover, $\mathfrak S J>0$ if the equation
$F(\bfx;\bfy)=0$ has a nonsingular solution over $\mathbb Q_p$ for every
prime $p$ and a nonsingular real solution.
\end{theorem}

\begin{proof}
The hypothesis and the definition of $K(F)$ give
\[
    K(F)>\mathcal C(d;w_1,w_2)
    =\max\{2,1/\eta_*\}.
\]
We may therefore apply Lemma \ref{lemma: minor arc bound}. Thus there
exist $\delta>0$ and $0<\vartheta_0<w_1/(5d)$ such that
\[
    \int_{\grm(\vartheta_0)}|S(\a)|\,\d\a
    \ll X^{s_1}Y^{s_2}P^{-1-\delta}.
\]
Since $K(F)>2$, Lemmas \ref{lemma: complete singular series} and
\ref{lemma: complete singular integral} imply
\[
    \grS(P^{\vartheta_0})=\grS+o(1),
    \qquad
    J(P^{\vartheta_0};1)=J+o(1).
\]
Applying Lemma \ref{lemma: circlemethod dissection} with
$\theta=\vartheta_0$, we obtain
\[
    R_F(P)
    =X^{s_1}Y^{s_2}P^{-1}\bigl(\grS J+o(1)\bigr).
\]
The asserted asymptotic formula follows from
\[
    X^{s_1}Y^{s_2}P^{-1}
    =P^{(s_1w_1+s_2w_2)/d-1}.
\]

Suppose finally that $F=0$ has a nonsingular solution over $\mathbb Q_p$
for every prime $p$ and a nonsingular real solution. Section
\ref{sec:major arcs} gives $\grS>0$. If $(\bfx_0;\bfy_0)$ is a
nonsingular real solution, then weighted homogeneity gives
\[
    F(\lambda^{w_1}\bfx_0;\lambda^{w_2}\bfy_0)=0
    \qquad(\lambda>0).
\]
The $\bfx$- and $\bfy$-partial derivatives are multiplied by
$\lambda^{d-w_1}$ and $\lambda^{d-w_2}$, respectively, so the resulting
solution remains nonsingular. Taking $\lambda$ sufficiently small places
it in $(-1,1)^{s_1+s_2}$. The discussion in Section
\ref{sec:major arcs} then gives $J>0$, whence $\grS J>0$.
\end{proof}

\begin{proof}[Proof of Theorem \ref{thm:V1-V2-criterion}]
The hypothesis of Theorem \ref{thm:V1-V2-criterion}, together with
\eqref{eq:C simple upper bound}, implies the hypothesis of Theorem
\ref{thm:V1-V2-criterion-precise}. The result follows.
\end{proof}

\subsection{Proof of Theorem \ref{thm:component-singular-locus}}

We first deduce a sharper statement from Theorem
\ref{thm:V1-V2-criterion-precise} and Lemma
\ref{lemma:codimension comparison}.

\begin{theorem}\label{thm:component-singular-locus-precise}
Suppose that $d>w_1w_2$ and that, for some $0\leq k\leq m$, one has
\[
    s_1+s_2-\sigma(k)
    >
    \mathcal C(d;w_1,w_2)2^{r_k}r_k.
\]
Then, as $P\to\infty$, one has
\[
    R_F(P)
    =\mathfrak S J\,
    P^{(s_1w_1+s_2w_2)/d-1}
    +o\left(P^{(s_1w_1+s_2w_2)/d-1}\right),
\]
where $\mathfrak S$ and $J$ are the singular series and singular integral
associated with $F$. Moreover, $\mathfrak S J>0$ if the equation
$F(\bfx;\bfy)=0$ has a nonsingular solution over $\mathbb Q_p$ for every
prime $p$ and a nonsingular real solution.
\end{theorem}

\begin{proof}
Put
\[
    G=F^{(d_1(k),d_2(k))}.
\]
Suppose first that $d_1(k)d_2(k)>0$. By the definitions of
$\sigma(k)$ and $\widehat\sigma(k)$ and by Lemma
\ref{lemma:codimension comparison},
\[
\begin{aligned}
    s_1+s_2-\widehat\sigma(k)
    &=\min_{i\in\{1,2\}}\codim{V_i(G)},\\
    &\geq\frac12\codim{\Sing{G}},\\
    &=\frac12\bigl(s_1+s_2-\sigma(k)\bigr).
\end{aligned}
\]
The hypothesis therefore gives
\[
    s_1+s_2-\widehat\sigma(k)
    >\mathcal C(d;w_1,w_2)2^{r_k-1}r_k,
\]
which is the hypothesis of Theorem
\ref{thm:V1-V2-criterion-precise} for this index.

If $d_1(k)d_2(k)=0$, then the inactive variables are unrestricted in the
singular locus of $G$. Hence
\[
    \sigma(k)=\widehat\sigma(k),
\]
and the hypothesis is again precisely the corresponding hypothesis in
Theorem \ref{thm:V1-V2-criterion-precise}. The result follows in every
case.
\end{proof}

\begin{proof}[Proof of Theorem \ref{thm:component-singular-locus}]
The result follows immediately from Theorem
\ref{thm:component-singular-locus-precise} and
\eqref{eq:C simple upper bound}.
\end{proof}

\subsection{Proof of Theorem \ref{thm:weighted-singular-locus}}
\label{sec:full singular locus reduction}

We first record the sharper statement obtained from Theorem
\ref{thm:component-singular-locus-precise}.

\begin{theorem}\label{thm:weighted-singular-locus-precise}
Suppose that $d>w_1w_2$ and
\[
    s_1+s_2-\sigma_F
    >
    \mathcal C(d;w_1,w_2)
    \sum_{k=0}^{m}2^{r_k}r_k.
\]
Then, as $P\to\infty$, one has
\[
    R_F(P)=\mathfrak S J\,P^{(s_1w_1+s_2w_2)/d-1}
    +o\left(P^{(s_1w_1+s_2w_2)/d-1}\right).
\]
Moreover, $\mathfrak S J>0$ whenever $F(\bfx;\bfy)=0$ has a nonsingular
real solution and a nonsingular solution over $\Q_p$ for every prime $p$.
In particular, when $\sigma_F=0$, the weighted hypersurface defined by
$F=0$ satisfies the integral Hasse principle under the stated hypotheses.
\end{theorem}

\begin{proof}
For $0\leq k\leq m$, put
\[
    C_k:=s_1+s_2-\sigma(k),
    \qquad
    T_k:=\mathcal C(d;w_1,w_2)2^{r_k}r_k,
\]
and let
\[
    W_k:=\Sing{F^{(d_1(k),d_2(k))}}.
\]
Each $W_k$ is a homogeneous affine variety. Choose an irreducible
component $W_k'$ of $W_k$ having dimension $\sigma(k)$. Since $W_k'$ is
an affine cone, it contains the origin. Hence $\bigcap_kW_k'$ is
nonempty, and the affine dimension theorem gives
\[
    \codim{\bigcap_kW_k'}
    \leq
    \sum_k\codim{W_k'}
    =\sum_kC_k.
\]
As $\bigcap_kW_k'\subseteq\bigcap_kW_k$, it follows that
\[
    \codim{\bigcap_kW_k}
    \leq\sum_kC_k.
\]
Moreover, the intersection of the singular loci of the bihomogeneous
components in \eqref{eqs:decomp} is contained in the singular locus of
$F$. Therefore
\[
    s_1+s_2-\sigma_F
    \leq\sum_{k=0}^{m}C_k.
\]
If $C_k\leq T_k$ for every $k$, this inequality would contradict the
hypothesis. Thus $C_{k_0}>T_{k_0}$ for some $0\leq k_0\leq m$, and
Theorem \ref{thm:component-singular-locus-precise} applies.
\end{proof}

\begin{proof}[Proof of Theorem \ref{thm:weighted-singular-locus}]
Put
\[
    e=\left\lfloor\frac d{w_1}\right\rfloor.
\]
The numbers $r_k$ are distinct positive integers. Moreover, the
definition of $d_2(0)$ in Section \ref{section4!} gives
\[
    r_k\leq r_0\leq e-1.
\]
Consequently,
\[
\sum_{k=0}^{m}2^{r_k}r_k \leq\sum_{r=1}^{e-1}2^rr =2^e(e-2)+2 <\frac d{w_1}2^{d/w_1}.
\]
Together with \eqref{eq:C simple upper bound}, this shows that the
hypothesis of Theorem \ref{thm:weighted-singular-locus} implies that of
Theorem \ref{thm:weighted-singular-locus-precise}. The result follows.
\end{proof}

\printbibliography  

\end{document}